\documentclass[a4paper,12pt]{amsart}

\usepackage{amssymb, manfnt}
\usepackage{hyperref}
\usepackage[normalem]{ulem}
\usepackage{tikz}
\usepackage{booktabs}
\usepackage{multirow}
\usepackage{MnSymbol}
\usepackage{graphicx}
\usepackage{tikz-cd}
\usepackage{longtable}
\usepackage{amsmath,amscd,amsthm}
\usepackage{tkz-graph}
\tikzstyle{vertex}=[circle, draw, inner sep=0pt, minimum size=6pt]

\usepackage[all]{xy}
\graphicspath{ {graphs/} }
\usetikzlibrary{decorations.text}
\usepackage{graphicx}
\usepackage[outdir=./pictures]{epstopdf}
\graphicspath{ {pictures/} }
\usepackage[all]{xy}
\usepackage{adjustbox}
\usepackage{graphicx}
\usepackage{fullpage}
\usepackage[T1]{fontenc}
\usepackage{ifthen}
\usepackage[all]{xy}
\usepackage{pdflscape}
\usepackage{array}
\usepackage{multirow}
\usepackage{rotating}

\def\altdb{\vadjust{\vbox to 0pt{\vss\hbox{\kern \hsize
\quad{\dbend}}\kern\baselineskip\kern-10pt}}}

\newcommand{\arxiv}[1]{\href{http://arxiv.org/abs/#1}{\tt arXiv:\nolinkurl{#1}}}
\newcommand{\doi}[1]{\href{http://dx.doi.org/#1}{{\tt DOI:#1}}}

\newcommand{\googlebooks}[1]{(preview at \href{http://books.google.com/books?id=#1}{google books})}

\makeatletter
\let\@@pmod\pmod
\DeclareRobustCommand{\pmod}{\@ifstar\@pmods\@@pmod}
\def\@pmods#1{\mkern4mu({\operator@font mod}\mkern 6mu#1)}
\makeatother

  \newcommand\subcat[1]{\underline{#1}}

\renewcommand\Vec{\operatorname{Vec }}
\newcommand\IVec{\operatorname{Vec }^{-}}
\newcommand\LVec{\operatorname{Vec }}
\newcommand\sVec{\operatorname{sVec }}

\newcommand\BrPic{\operatorname{BrPic}}

\newcommand\id{\operatorname{id}}

\newcommand\TenAut{\operatorname{Eq}}
\newcommand\Rep{\operatorname{Rep}}

\newcommand\Id{\operatorname{Id}}
\newcommand\Hom{\operatorname{Hom}}

\newcommand\ad{\operatorname{Ad}}

\newcommand\Z[1]{ \mathbb{Z}_{#1}}

\newcommand\cC{ \mathcal{C}}
\newcommand\cM{ \mathcal{M}}
\newcommand\cZ{ \mathcal{Z}}
\newcommand\cF{ \mathcal{F}}

\theoremstyle{plain}
\newtheorem{theorem}{Theorem}[section]
\newtheorem*{theorem*}{Theorem}
\newtheorem*{prop*}{Proposition}
\newtheorem{cor}[theorem]{Corollary}
\newtheorem{lemma}[theorem]{Lemma}
\newtheorem{prop}[theorem]{Proposition}

\newtheorem{rmk}[theorem]{Remark}
\theoremstyle{remark}

\theoremstyle{definition}
\newtheorem{dfn}[theorem]{Definition}

\newcommand\Inv{\operatorname{Inv}}

\usepackage{color}

\numberwithin{equation}{section}

\DeclareRobustCommand*{\nicefrac}{\@UnitsNiceFrac}%
\makeatother

\usepackage{longtable}
\title{Classifying fusion categories $\otimes$-generated by an object of small Frobenius-Perron dimension}

\author{Cain Edie-Michell}
\address{Cain Edie-Michell\\
Department of Mathematics, Vanderbilt University\\
Nashville\
USA}
\email{cain.edie-michell@vanderbilt.edu}

\allowdisplaybreaks
\begin{document}

\maketitle

\begin{abstract}
The goal of this paper is to classify fusion categories $\cC$ which are $\otimes$-generated by an object $X$ of Frobenius-Perron dimension less than 2, with the additional mild assumption that the adjoint subcategory of $\cC$ is $\otimes$-generated by the object $X\otimes X^*$. This classification has recently become accessible due to a result of Morrison and Snyder, showing that any such category must be a cyclic extension of a category of adjoint $ADE$ type. Our main tools in this classification are the results of \cite{MR2677836}, classifying cyclic extensions of a given category in terms of data computed from the Brauer-Picard group, and Drinfeld centre of that category, and the results of \cite{MR3808052} which compute the Brauer-Picard group and Drinfeld centres of the categories of adjoint $ADE$ type.
 
Our classification includes the expected categories, constructed from cyclic groups and the categories of $ADE$ type. More interestingly we have categories in our classification that are non-trivial de-equivariantizations of these expected categories. Most interesting of all, our classification includes three infinite families constructed from the exceptional quantum subgroups $\mathcal{E}_4$ of $\cC( \mathfrak{sl}_4, 4)$, and $\mathcal{E}_{16,6}$ of $\cC( \mathfrak{sl}_2, 16)\boxtimes\cC( \mathfrak{sl}_3,6)$.
\end{abstract}

\section{Introduction}\label{sec:intro}
Fusion categories are a natural generalization of the representation category of a finite group, where we now allow the tensor product to be non-commutative. In this sense one can think of the program to classify fusion categories as the natural successor to the program to classify finite groups. While the classification of finite simple groups has been completed, the classification of fusion categories is still far from complete. Currently there are not even conjectures for a classification statement of all fusion categories. However, it seems reasonable to expect the existence of several truly exotic fusion categories, analogous to the situation with finite simple groups.

As a complete classification of fusion categories is hopelessly out of reach with current techniques, current research into the classification of fusion categories focuses on classifying ``small'' fusion categories, where small can have a variety of different meanings. Examples of such partial classifications can be found in \cite{1309.4822} where a classification of pivotal fusion categories with exactly three simple objects is given, or in \cite{MR3624901} where a classification of pivotal fusion categories with restrictions on the size of certain hom spaces is found.

Besides being interesting purely for their rich algebraic structure, fusion categories are important due to their relationship with several other areas of mathematics and physics. More precisely fusion categories provide a unifying framework for operator algebras, representation theory, and quantum field theory. Examples of fusion categories in these subjects appear as; the even part of a finite depth subfactor, the category of level $k$ integrable representations of an affine Lie algebra, and the value of a point in a fully extended $2+1$ dimensional topological quantum field theory. Thus partial classification results for fusion categories have broad applications to these subjects.

This paper will add another partial classification result to the literature. For us a ``small'' category will mean one $\otimes$-generated by an object of small Frobenius-Perron dimension. This notion of small is not new, and can be traced back to the earliest days of subfactor theory. Attempts to partially classify such small fusion categories have proven particularly successful in constructing exotic examples, such as the extended Haagerup fusion category \cite{MR2979509} and Izumi's quadratic categories \cite{MR1832764}. These examples remain the only known fusion categories not yet shown to be related to finite or quantum groups. One of the motivations behind this paper was to find new exotic examples appearing in our partial classification. Instead we find that every category appearing in our classification can be directly constructed from finite or quantum groups, though sometimes in very interesting and non-trivial ways! Thus our main Theorem provides further evidence that exotic fusion categories are indeed very rare objects.

The main Theorem of this paper is a generalization of two existing partial classifications. The first is the $ADET$ classification of unitary fusion categories $\otimes$-generated by a self-dual object of dimension less than 2. This result is closely related to the famous classification of subfactors of index less than 4 \cite{MR1193933,MR1145672,MR1313457,MR1929335,MR1308617, MR1617550}. The second is the classification of braided fusion categories $\otimes$-generated by an object of Frobenius-Perron dimension less than 2 \cite{MR1239440}. Our result vastly generalises both of these results. We classify fusion categories $\cC$ which are $\otimes$-generated by an object $X$ of Frobenius-Perron dimension less than 2, with the assumption that the adjoint subcategory of $\cC$ is $\otimes$-generated by the object $X\otimes X^*$.

While the assumption that the adjoint subcategory of $\cC$ is $\otimes$-generated by the object $X\otimes X^*$ may seem excessively restrictive, it is actually quite a mild assumption. It generalises the condition that $\otimes$-generating object $X$ be self-dual, and even generalises the weaker condition of $X$ commuting with its dual (which we have if $\cC$ has a braiding, or if the fusion ring is commutative). In fact, to the authors best knowledge, the only known examples of fusion categories not satisfying this assumption are the wreath product categories found in \cite{1904.08909}.

\begin{theorem}\label{thm:main}
Let $\cC$ be a fusion category $\otimes$-generated an object of Frobenius-Perron dimension less than 2, such that $\ad(\cC) = \langle X \otimes X^* \rangle$. Then, up to twisting the associator of $\cC$ by an element of $H^3(\text{Grading Group}, \mathbb{C}^\times)$, the category $\cC$ is monoidally equivalent to one of the following:
\begin{center}
  \begin{longtable}{l | c c  }
    	\toprule
			Category        					         									  & Parameterisations									 & Grading group     \\
	\midrule
			$\Vec( \Z{M})$	 					     		        					  	  & 				     							           & $\Z{M}$ \\[.35cm]
			$\ad(A_{2N}^{(n)}) \boxtimes \Vec( \Z{M})$   	   								 &  $n \in  \Z{2N+1}^\times / \{\pm 1\}$ 					         & $\Z{M}$ \\[.35cm]
		 	$\subcat{A_{2N+1}^{(n)}\boxtimes \Vec( \Z{2M} )  }$   								 &  $n \in  \Z{2N+2}^\times / \{\pm 1\}$					         & $\Z{2M}$ \\[.35cm]
	   		$\subcat{A_{4N+1}^{(n)}\boxtimes \Vec( \Z{4M} )  }_{ \langle f^{(4N)} \boxtimes 2M  \rangle    }$      &   $n \in  \Z{4N+2}^\times / \{\pm 1\}$			         & $\Z{2M}$ \\[.35cm]
	   		$\subcat{A_{4N+3}^{(n)}\boxtimes \IVec( \Z{4M} )  }_{ \langle f^{(4N+2)} \boxtimes 2M  \rangle    }$      &   $n \in  \Z{4N+4}^\times / \{\pm 1\}$			         & $\Z{2M}$ \\[.35cm]
	   		$\subcat{A_{3}^{(1)}\boxtimes \IVec( \Z{8M} )  }_{ \langle f^{(2)} \boxtimes 4M  \rangle    }$     		 & 											& $\Z{4M}$ \\[.35cm]
	   		$\subcat{D_{2N}^{(n,\pm)}\boxtimes \Vec( \Z{2M} )  }$		            					& $n \in  \Z{4N-2}^\times / \{\pm 1\}$			       & $\Z{2M}$ \\[.35cm]
	   		$\subcat{D_{4}^{(1,\pm)}\boxtimes \Vec( \Z{18M} )  }_{ \langle P \boxtimes 6M \rangle    }$     	   &  											         & $\Z{6M}$ \\[.35cm]
	   		$\subcat{E_{6}^{(n,\pm)}\boxtimes \Vec( \Z{2M} )}  $		         					    &  $n \in  \Z{12}^\times / \{\pm 1\}$						         & $\Z{2M}$ \\[.35cm]
	   		$\subcat{E_{6}^{(n,\pm)}\boxtimes \IVec( \Z{4M} )}_{ \langle Z \boxtimes 2M  \rangle    } $     	   &  $n \in  \Z{12}^\times / \{\pm 1\}$						         & $\Z{2M}$ \\[.35cm]
	   		$\subcat{E_{8}^{(n,\pm)}\boxtimes \Vec( \Z{2M} )}$    			 					  &  $n \in  \Z{30}^\times / \{\pm 1\}$						         & $\Z{2M}$ \\[.35cm]
	   		$\subcat{\mathcal{E}_4^{(n)}\boxtimes \Vec( \Z{4M} )}$    							  &  $n \in  \Z{8}^\times / \{\pm 1\}$						         & $\Z{4M}$ \\[.35cm]
	   		$\subcat{\mathcal{E}_4^{(n)}\boxtimes \IVec( \Z{16M} )}_{ \langle \mathbf{4} \boxtimes 8M  \rangle    } $    &  $n \in  \Z{8}^\times / \{\pm 1\}$						         & $\Z{8M}$ \\[.35cm]
	   		$\subcat{\mathcal{E}_{16,6}^{(n,\pm)}\boxtimes \Vec( \Z{6M} )} $         					 &  $n \in  \Z{18}^\times / \{\pm 1\}$						         & $\Z{6M}$ \\[.35cm]
    	\bottomrule
    \end{longtable} 
    \end{center}

\end{theorem}

In order to make the above Theorem as self-contained as possible we provide a quick index of where to find relevant information on the categories and constructions used in the statement of this Theorem.
 
\begin{table}[h!]
\centering 
    \begin{tabular}{l | l}
	\midrule
			Twisting the associator by an element of $H^3(\text{grading group},\mathbb{C}^\times)$ & Definition~\ref{def:twisting}\\[.1cm]
 		       Categories of $ADE$ type  & Subsection~\ref{sub:ADE}	\\[.1cm]
 		       $\IVec(\Z{2M})$  & Definition~\ref{def:2Mtw}	\\[.1cm]
 		     $\subcat{\cC}$  & Definition~\ref{def:subcat}	\\[.1cm]
 		     $\cC_{\Rep(G)}$  & Definition~\ref{def:deeq}	\\[.1cm]
		     $\mathcal{E}_4^{(n)}$ & Definition~\ref{def:e4}\\[.1cm]
		    $\mathcal{E}_{16,6}^{(n,+)}$ & Definition~\ref{def:e166}\\[.1cm]
		  Fusion rules for $\mathcal{E}^{(n)}_4$ and $\mathcal{E}^{(n,\pm)}_{16,6}$    & Appendix~\ref{app:rules}  \\
    	\bottomrule
    \end{tabular}
\end{table}

We roughly sum up our classification as follows. Any fusion category appearing in our classification is directly constructed from a cyclic pointed category, a category of $ADE$ type, or from one of the quantum subgroups $\mathcal{E}_4$ or $\mathcal{E}_{16,6}$.

From an operator algebraic perspective it is interesting to know which of the categories in Theorem~\ref{thm:main} are unitary. This was worked out in the authors Thesis \cite{cainthesis}, where additional details can be found. In the unitary setting, we lose the choice of $n$, which must always be equal to $1$. An interesting application of the unitary version of Theorem~\ref{thm:main} is that when paired with Popa's embedding theorem, we get a classification of certain bimodules of the hyperfinite type $II_1$ factor, with Murray-Von Neumann dimension less than 2.

The reader may find it unsatisfying that we only classify categories up to twisting the associator by some $3$-cocycle. To ease the readers mind, we direct them to Lemma~\ref{lem:twtwt}, which provides an explicit recipe for constructing a cocycle twist of a fusion category.

It is also important to note that while each of the fusion categories in Theorem~\ref{thm:main} are monoidally inequivalent, even up to twisting the associator by a $3$-cocyle, it is not true that for a fixed category in our classification result that all $3$-cocyle twists will be monoidally inequivalent. For example, the 5 cocycle twists of $\Vec(\Z{5})$ only give 3 monoidally inequivalent fusion categories. While we have tried hard to refine our result, such a problem proved far beyond the techniques developed in \cite{1711.00645} toward this goal.

The structure of this paper is as follows: Section~\ref{sec:prelim} contains the necessary background material to understand the statement of Theorem~\ref{thm:main}, along with the tools and machinery to prove this Theorem. In particular we communicate a Theorem of Morrison and Snyder, showing that any fusion category $\cC$, which is $\otimes$-generated by an object of Frobenius-Perron dimension less than 2, such that $\ad(\cC) = \langle X\otimes X^* \rangle$, must be a cyclic extension of a category of adjoint $ADE$ type. With this Theorem in mind, we spend Section~\ref{sec:class} classifying such extensions of the categories of adjoint $ADE$ type. Key for these computations were the Authors results \cite{MR3808052} computing the Drinfeld centres, and Brauer-Picard groups of the categories of adjoint $ADE$ type. Section~\ref{sec:mainproof} ties the results of Section~\ref{sec:class} together in order to give a proof of Theorem~\ref{thm:main}. In an appendix to this paper, we compute the fusion rules for the categories $\mathcal{E}^{(n)}_4$ and $\mathcal{E}^{(n,\pm)}_{16,6}$.  While these categories come from well known conformal inclusions, the fusion rules for these categories have not been computed in the literature before. We exploit the fact that these categories are graded extensions of categories which we know the fusion rules for, to compute the fusion rules for the entire category.

A natural generalization of Theorem~\ref{thm:main} is to increase the bound on the Frobenius-Perron dimension of the $\otimes$-generating object. If we assume a unitary condition on our fusion categories, then it seems feasible to increase the bound the dimension of the $\otimes$-generating object from $4$ up to $\sqrt{5 + \frac{1}{4}}$. Such a category would have to be a cyclic extension of the even part of a finite depth subfactor of index less than $5 + \frac{1}{4}$, which have been completely classified in \cite{1509.00038}. Furthermore the Brauer-Picard groups of the even parts of many of these subfactors have been computed by Grossman and Snyder \cite{MR3449240,MR2909758}. With the information of the Brauer-Picard groups given, the techniques of this paper should directly generalize to prove such a classification.

Another interesting future line of research would be to generalise Theorem~\ref{thm:main} by removing the condition that $\ad(\cC)  =\langle X\otimes X^* \rangle$, to obtain a complete classification of fusion categories $\otimes$-generated by an object of Frobenius-Perron dimension less than $2$. Partial progress towards this goal has been made in \cite{1904.08909}, where a classification of pivotal fusion categories $\otimes$-generated by an object of dimension $\frac{1 +\sqrt{5}}{2}$ was given. Given the appearance of the interesting categories $\mathcal{E}_4$ or $\mathcal{E}_{16,6}$ in our weaker classification, we anticipate even more interesting categories appearing in such a generalised classification.

\subsection*{Acknowledgments}
The author is grateful to many people for their help with this paper. This paper is an adaptation of the authors PhD Thesis, supervised by Scott Morrison who provided invaluable support. I would also like to thank Pinhas Grossman, Corey Jones, and Noah Synder for interesting discussions, and Marcel Bischoff, Victor Ostrik, and Eric Rowell for their comments on the Thesis version of this paper.

This research was supported by an Australian Government Research Training Program (RTP) Scholarship. The author was partially supported by the Discovery Project 'Subfactors and symmetries' DP140100732 and 'Low dimensional categories' DP160103479.

\section{Preliminaries}\label{sec:prelim}
A \textit{fusion category} is a finite semisimple $\mathbb{C}$-linear rigid monoidal category with simple unit.

Let $X$ be an object in a fusion category $\cC$. We define the \textit{Frobenius-Perron dimension} of $X$ as the unique largest real eigenvalue of the matrix that represents tensoring the simple objects of $\cC$ with $X$. We have that the Frobenius-Perron dimension of an object $X$ is always greater than or equal to 1. Furthermore if the Frobenius-Perron dimension of an object $X$ is less than 2, then it can only take values in the countable set
\[  \left\{ 2\cos\left( \frac{\pi}{n}\right) : n \in \mathbb{N} \right\}  .\]
We $\cC$ is \textit{pivotal} if it is equipped with a monoidal equivalence
\[  \Id_{\cC} \to **.\] 
If $\cC$ is pivotal then we can define the \textit{categorical dimension} of a simple object $X$ as the trace of the identity morphism on $X$. Like the Frobenius-Perron dimension, the categorical dimension of $X$ is also an eigenvalue of the matrix that represents tensoring the simple objects of $\cC$ with $X$. However now the categorical dimension of $X$ can be close to 0, or even negative. 

We say a fusion category $\cC$ is $\otimes$-generated by an object $X$ if every object $Y\in \cC$ is a sub-object of some tensor power of $X$, or equivalently, if the fusion graph for tensoring by $X$ is connected. Given $X \in \cC$, we write $ \langle X \rangle$ for the fusion subcategory of $\cC$ $\otimes$-generated by $X$.

\subsection{De-equivariantization}
A key tool for constructing the categories in the main classification result of this paper is de-equivariantization. 

Let $\cC$ a fusion category, $\mathcal{D}$ a monoidal subcategory of $\cC$, and $\mathcal{E}$ a braided fusion category. We say \textit{$\mathcal{D}$ lifts to a copy of $\mathcal{E}$ in the centre of $\cC$} if there exists a fully faithful functor $\mathcal{F}: \mathcal{D} \to \cZ(\cC)$ such that $\mathcal{F}(\mathcal{D})$ is braided equivalent to $\mathcal{E}$, and the following diagram commutes 
\begin{center}
\begin{tikzcd}
 & \cZ(\cC) \arrow[d] \\
\mathcal{D} \arrow[r,hook]  \arrow[ru,hook] & \cC
\end{tikzcd}
\end{center}
where the functor $\cZ(\cC) \to \cC$ is the forgetful functor.

Let $G$ a finite group, and $\mathcal{D}$ be a monoidal subcategory of $\cC$ with a lift to a copy of $\Rep(G)$ in the centre. Then we can consider the function algebra object $A := \operatorname{Fun}(G)$ in $\mathcal{D}$. As $\mathcal{D}$ has a lift to a copy of $\Rep(G)$ in the centre of $\cC$, we have that the algebra $A \in \cC$ has a commutative half-braiding. Thus the category of $A$-modules in $\cC$ has the structure of a fusion category by \cite{MR2863377}. 
\begin{dfn}\label{def:deeq}
Let $\mathcal{D}$ be a monoidal subcategory of $\cC$ with a lift to a copy of $\Rep(G)$ in the centre. We define $\cC_{\mathcal{D}}$, the de-equivariantization of $\cC$ by $\mathcal{D}$, as the fusion category of $A$-modules in $\cC$.
\end{dfn}

The following categories provide examples of categories with a subcategory that lifts to a copy of $\operatorname{sVec}$ in the centre. While we can't directly de-equivariantize these categories, they will be useful in constructions that we can de-equivariantize.

\begin{dfn}\label{def:2Mtw}
We write $\IVec(\Z{2M})$ for the category with $\Vec(\Z{2M})$ fusion rules, and associator given by 
\[   \alpha_{n_1,n_2, n_3} := \begin{cases}
1, \quad \hspace{.75em} \text{ if } n_2 + n_3 < n_1,\\
e^{2i\pi\frac{n_1}{M}} \text{ if } n_2 + n_3 \geq n_1.
\end{cases}\]
\end{dfn}

\begin{lemma}\label{lem:liftyourself} 
\hspace{1em}
\begin{enumerate}
\item The subcategory $\langle M \rangle \subseteq \Vec(\Z{2M})$ lifts to the centre as a copy of $\Rep(\Z{2})$.
\item The subcategory $\langle M \rangle \subseteq \IVec(\Z{2M})$ lifts to the centre as a copy of $\operatorname{sVec}$.
\end{enumerate}
\end{lemma}
\begin{proof}
Recall that pointed braided fusion categories are classified by pairs $(G,t)$, where $G$ is an abelian group, and $t:G\to \mathbb{C}^\times$ is a quadratic form. In the case $G =\Z{2M}$ we have that the quadratic form $t$ is completely determined by $q$ a $4M^2$-th root of unity, with 
\[  t_q(m) := q^{m^2}.\]
In particular we have the two pointed braided fusion categories corresponding to the $4M^2$-th roots of unity 1 and $e^{2\pi i \frac{1}{2M^2}}$. The twist of the object $M$ in each of these categories is given by $t_1(M) = 1$ and $t_{e^{2\pi i \frac{1}{2M^2}}}(M)= -1$ respectively. Thus we get braided inclusions $\Rep(\Z{2}) \to (\Z{2M} , t_1)$ and $\operatorname{sVec} \to (\Z{2M} ,t_{e^{2\pi i \frac{1}{2M^2}}})$. A braided category $\cC$ has a canonical embedding $\cC \to \cZ(\cC)$, hence the subcategory $\langle M \rangle \subseteq (\Z{2M} , t_1)$ lifts to the centre as a copy of $\Rep(\Z{2})$, and the subcategory $\langle M \rangle \subseteq (\Z{2M} ,t_{e^{2\pi i \frac{1}{2M^2}}})$ lifts to the centre as a copy of $\operatorname{sVec}$. We use \cite[Section 2.5.2]{MR1734419} to see that the category $(\Z{2M} , t_1)$ is monoidally equivalent to $\Vec(\Z{2M})$, and the category $ (\Z{2M} ,t_{e^{2\pi i \frac{1}{2M^2}}})$ is monoidally equivalent to the category $\IVec(\Z{2M})$.
\end{proof}

\subsection{Bimodule categories}
Here we define module categories, bimodules categories, and the Brauer-Picard group.
\begin{dfn}\cite{MR1976459}
A left module category $\cM$ over a fusion category $\cC$ is a semi-simple $\mathbb{C}$-linear category along with a functor $ \otimes : \cC\times \cM \to \cM$, and natural isomorphisms $(X\otimes Y) \otimes M \to X \otimes (Y \otimes M)$ satisfying a straightforward pentagon equation. 
\end{dfn}

A slight generalisation of a module category over $\cC$, is the notion of a bimodule category over $\cC$. This is a natural generalisation where now the category $\cC$ can act on both the left and right, and there is the additional structure of an isomorphism relating the left and right actions (see \cite{MR2678824} for an explicit definition). 

Given a $\cC$ bimodule and a monoidal auto-equivalence $\cF \in \TenAut(\cC)$, one can construct a new bimodule by twisting the action on one side.
\begin{dfn}\label{dfn:twist}
Let $\cM$ be a $\cC$-bimodule category, and $\cF \in \TenAut(\cC)$ a monoidal auto-equivalence. We define a new bimodule $_{\cF}\cM$, which is equal to $\cM$ as a right module category, and with left action given by 
\[    X \triangleright_{_{\cF}\cM} m   :=  \cF(X) \triangleright_M m .\]        
The structure morphisms for $_{\cF}\cM$ consist of a combination of the structure morphisms for $\cM$, and the tensorator of $\cF$.   
\end{dfn}

Given two $\cC$-bimodules we can define their relative tensor product over $\cC$, which is a another $\cC$-bimodule category. The details on this relative tensor product can be found in \cite{MR2677836}, however these detail are unnecessary to follow this paper. Using this relative tensor product of bimodules we can define the Brauer-Picard group of $\cC$.
\begin{dfn}
The Brauer-Picard group of $\cC$, which we denote $\BrPic(\cC)$, is the group of invertible $\cC$-bimodules with respect to the relative tensor product.
\end{dfn}

\subsection*{Graded categories}
Let $\cC$ a fusion category and $G$ a finite group. We say $\cC$ is a \textit{$G$-graded fusion category} if we can write
\[ \cC \simeq \bigoplus_G \cC_g,\]
with $\cC_g$ non-trivial abelian subcategories, such that the tensor product of $\cC$ restricted to $\cC_g \times \cC_h$ has image contained in $\cC_{gh}$.

We say $\mathcal{D}$ is \textit{$G$-graded extension} of a fusion category $\cC$ if $\mathcal{D}$ is a $G$-graded fusion category whose trivially graded piece is $\cC$. We define the adjoint subcategory of a fusion category $\cC$ as
\[ \ad(\cC) := \langle X \otimes X^*: X\in \operatorname{Irr}(\cC) \rangle.\] 
Every fusion category $\cC$ is a $G$-graded extension of $\ad(\cC)$ for some finite group $G$.

Let 
\[  \mathcal{D}_1  := \bigoplus_G \mathcal{D}_{1,g} \text{ and }  \mathcal{D}_2  := \bigoplus_G \mathcal{D}_{2,g} \]
be two $G$-graded extensions of $\cC$, so $\mathcal{D}_{1,e} = \cC = \mathcal{D}_{2,e}$. An equivalence of extensions is a monoidal equivalence $\mathcal{F}:\mathcal{D}_1 \to \mathcal{D}_2$, such that $\mathcal{F}_{\cC}$ is the identity, and $\mathcal{F}( \mathcal{D}_{1,g} ) \subseteq \mathcal{D}_{2,g}$. Equivalence of extensions is, in general, a stronger condition than plain monoidal equivalence. We will see several examples of this later in the paper.

Given a $G$-graded fusion category $ \cC \simeq \bigoplus_G \cC_g$, and an 3-cocyle $\omega \in H^3(G, \mathbb{C}^\times)$, we can construct a new fusion category $\cC^\omega$ with the same objects and fusion rules, but whose associator is given by
\[   \alpha^{\cC^\omega}_{X_f,Y_g,Z_h} := \omega(f,g,h)  \alpha^{\cC}_{X_f,Y_g,Z_h} .\]
\begin{dfn}\label{def:twisting}
We say two $G$-graded categories $\cC_1$ and $\cC_2$ are equivalent, up to twisting the associator, if there exists a 3-cocycle $\omega \in H^3(G, \mathbb{C}^\times)$ such that there exists a monoidal equivalence
\[  \cC_1 \simeq \cC_2^\omega.\]
\end{dfn}
We remark that the action of $H^3(G, \mathbb{C}^\times)$ on the set of $G$-graded categories is not generally free. For example if we take the $\Z{5}$-graded category $\Vec(\Z{5})$, and $\omega$ any non-trivial element of $H^3(\Z{5}, \mathbb{C}^\times)$, then the categories $\Vec(\Z{5})^\omega$ and $\Vec(\Z{5})^{\omega^{-1}}$ are monoidally equivalent (though inequivalent as $\Z{5}$-graded extensions of $\Vec$).

In this paper we will be mainly concerned with $\Z{M}$-graded categories. Such categories have many nice properties and constructions.

\begin{dfn}\label{def:subcat}
Let $\cC$ a $\Z{M_1}\times \Z{M_2}$-graded fusion category. We define $\subcat{\cC}$ as the subcategory $\otimes$-generated by the objects in $\cC_{1\times1}$.
\end{dfn}
It is straightforward to check that $\ad(\cC) = \ad(\subcat{\cC})$, thus $\cC$ and $\subcat{\cC}$ are extensions of the same category, though with possibly different grading groups.

This construction allows us to construct new cyclic-graded extensions of $\cC$ from old cyclic-graded extensions.
\begin{lemma}
Let $\mathcal{D}$ a $\Z{M}$-graded extension of $\ad(\mathcal{D})$. Then the category 
\[    \subcat{\mathcal{D} \boxtimes \Vec(\Z{N})} \]
is a $\Z{L}$-graded extension of $\ad(\mathcal{D})$, where $L$ is the least common multiple of $M$ and $N$.
\end{lemma}

Given a $\Z{M}$-graded fusion category $\cC$, and a cocycle $\omega \in H^3(\Z{M},\mathbb{C}^\times)$, we can use the above construction to realise the category $\cC^\omega$ in a concrete manner.
\begin{lemma}\label{lem:twtwt}
There exists an equivalence of extensions 
\[  \cC^\omega \to \subcat{\cC \boxtimes \Vec^\omega(\Z{M})}.\]
\end{lemma}
\begin{proof}
We construct a fully faithful functor 
\[  \cC^\omega \to \cC \boxtimes \Vec^\omega(\Z{M})\]
by 
\[  X_i \mapsto  X_i \boxtimes i.\]
It is routine to check this functor is monoidal, and that the image of this functor exactly lies in $\subcat{\cC \boxtimes \Vec^\omega(\Z{M})}$. Thus this functor gives a monoidal equivalence
\[ \cC^\omega \to \subcat{\cC \boxtimes \Vec^\omega(\Z{M})}.  \] 
This equivalence preserves the grading group, and is the identity on the trivially graded piece, hence it is an equivalence of extensions.
\end{proof}
The above Lemma shows that classifying $\Z{M}$-graded categories, up to twisting the associator, is really no different to classifying $\Z{M}$-graded categories on the nose, as we can concretely construct all cocyle twists. We leave such cocycle twists out of the statement of our main classification theorem as it makes the statement of the theorem very messy. 

Given a fusion category $\cC$, the results of \cite{MR2677836} give a classification of $\Z{M}$-graded extensions of $\cC$, up to twisting the associator. The main ingredients of the classification are $\BrPic(\cC)$ the Brauer-Picard group of $\cC$, and $\Inv(\cZ(\cC))$ the group of invertible elements in the centre of $\cC$, along the knowledge of how $\BrPic(\cC)$ acts on $\Inv(\cZ(\cC))$.

We have that $\Z{M}$-graded extensions of $\cC$, up to twisting the associator, are classified by a tuple $(c,T)$, where 
\[  c : \Z{M} \to \BrPic(\cC),\]
is a group homomorphism such that a certain obstruction $o_3(c) \in H^3(\Z{M} , \Inv(\cZ(\cC)))$ vanishes, and $T$ is an element of the group $H^2(\Z{M} , \Inv(\cZ(\cC)))$. We are purposely light on the details of this classification, as we need surprising little knowledge of this extension theory for this paper. Additional information can be found in the papers \cite{MR2677836} and \cite{1711.00645}. The main takeaway of this classification is that for a fixed homomorphism $ c : \Z{M} \to \BrPic(\cC)$, the order of the group $H^2(\Z{M} , \Inv(\cZ(\cC)))$ provides an upper bound for the number of $\Z{M}$-graded extensions, up to twisting the associator, realising $c$. An important detail we need to know regarding the extension theory of graded categories is that any extension corresponding to the homomorphism $c$ is equivalent to the category
\[   \bigoplus_{n\in \Z{M}} c(n)\]
as a bimodule category over $\cC$. This in particular implies that the Frobenius-Perron dimensions of the objects in an the extension corresponding to the homomorphism $c$ are completely determined by $c$.

Somewhat annoyingly for the purposes of this paper, the above classification of $\Z{M}$-graded extensions of $\cC$ is only up to equivalence of extensions, and not monoidal equivalence. This issue has been somewhat rectified in \cite{1711.00645}, where it is shown that in order to get a representative from each monoidal equivalence class, one only needs to consider homomorphisms
\[  c : \Z{M} \to \BrPic(\cC),\]
up to post-composition by the inner automorphisms of $\BrPic(\cC)$ induced by the invertible bimodules $_\mathcal{F}\cC$, where $\mathcal{F}\in \TenAut(\cC)$. Assuming some additional restrictive conditions on $\cC$, one also only has to consider 2-cocycles $T\in H^2(\Z{M} , \Inv(\cZ(\cC)))$, up to action by the bimodules $_\mathcal{F}\cC$ (using the specified action of $\BrPic(\cC)$ on $\Inv(\cZ(\cC))$. These two results will be important in refining our main classification theorem.

\subsection{Fusion categories of $ADE$ type}\label{sub:ADE}
Here we define the \textit{fusion categories of $ADE$ type}, which are key objects for this paper. The importance of these categories is that any fusion category $\otimes$-generated by an object of Frobenius-Perron dimension less than 2, must be a $\Z{M}$-graded extension of the adjoint subcategory of a fusion category of $ADE$ type. Further the fusion categories of $ADE$ type are examples of fusion categories $\otimes$-generated by an object of Frobenius-Perron dimension less than 2. A large portion of the categories appearing in our classification statement are constructed from the fusion categories of $ADE$ type.

\begin{trivlist}\leftskip=2em
\item \textbf{Fusion categories of $A_N$ type}:

For a fixed $N$, fusion categories of $A_N$ type are classified by $n \in \Z{N+1}^\times$. We label these categories $A_N^{(n)}$. Let $q = e^{\frac{2 i \pi n}{2N+2}}$. The categories $A_N^{(n)}$ can be realised as the idempotent completion of the $A_N$ planar algebra with loop parameter $q + q^{-1}$, or alternatively as the semisimplification of the category $\Rep(U_{-q}(\mathfrak{sl}_2))$. The categories $A_N^{(n)}$ have $N$ simple objects, which we label $f^{(m)}$ for $0\leq m < N$. The categories $A_N^{(n)}$ all have the same fusion rules, which can be found in \cite{AN-survey} under the translation $f^{(m)} \to X_m$.

\vspace{1em}

\item \textbf{Fusion categories of $D_{2N}$ type}:

For a fixed $N$, fusion categories of $D_{2N}$ type are classified by $n \in \Z{4N-3}^\times$, along with a choice of sign $\pm$. We label these categories $D_{2N}^{(n,\pm)}$. Let $q = e^{\frac{2 i \pi n}{8N-6}}$. The categories $D_{2N}^{(n,\pm)}$ can be realised as the idempotent completion of the $D_{2N}$ planar algebra with loop parameter $q + q^{-1}$ and rotational eigenvalue of $S$ given by $\pm i$, or alternatively as a certain de-equivariantization of the category $\Rep(U_{-q}(\mathfrak{sl}_2))$. The categories $D_{2N}^{(n,\pm)}$ have $2N$ simple objects, which we label $f^{(m)}$ for $0\leq m < 2N-2$ , $P$, and $Q$. The categories $D_{2N}^{(n,\pm)}$ all have the same fusion rules which can be found in \cite[Section 7]{MR1936496}, under the translation $f^{(m)} \to X_m$, $P\to X_{2N}^+$, and $Q\to X_{2N}^-$.

\vspace{1em}

\item \textbf{Fusion categories of $E_6$ type}:

The fusion categories of $E_6$ type are classified by $n\in \Z{12}^\times = \{1,5,7,11\}$, along with a choice of sign $\pm$. We label these categories $E_6^{(n,\pm)}$. Let $q = e^{\frac{2 i \pi n}{24}}$. The categories $E_6^{(n,\pm)}$ can be realised as the idempotent completion of the $E_6$ planar algebra with loop parameter $q + q^{-1}$ and rotational eigenvalue of $S$ given by $\pm e^{\frac{2\pi i 2}{3}}$, or alternatively as a quantum subgroup of the category $\Rep(U_{-q^{\pm 1}}(\mathfrak{sl}_2))$. The categories $E_{6}^{(n,\pm)}$ have $6$ simple objects, which we label $f^{(0)}$, $f^{(1)}$, $f^{(2)}$, $X$, $Y$, and $Z$. The categories $E_6^{(n,\pm)}$ all have the same fusion rules which can be found in \cite[Section 3.3]{MR1145672}, under the translation $f^{(0)} \to \text{id}$, $f^{(1)} \to \rho_1$, $f^{(2)} \to \rho_2$, $X \to \rho_3$, $Y \to \alpha \rho_1$, and $Z \to \alpha$.

\vspace{1em}

\item \textbf{Fusion categories of $E_8$ type}:

The fusion categories of $E_8$ type are classified by $n \in \Z{30}^\times = \{1,7,11,13,17,19,23,29\}$, along with a choice of sign $\pm$. We label these categories $E_8^{(n,\pm)}$. Let $q = e^{\frac{2 i \pi n}{60}}$. The categories $E_8^{(n,\pm)}$ can be realised as the idempotent completion of the $E_8$ planar algebra with loop parameter $q + q^{-1}$ and rotational eigenvalue of $S$ given by $\pm e^{\frac{2\pi i 3}{5}}$, or alternatively as a quantum subgroup of the category $\Rep(U_{-q^{\pm 1}}(\mathfrak{sl}_2))$. The categories $E_{8}^{(n,\pm)}$ have $8$ simple objects, which we label $f^{(0)}$, $f^{(1)}$, $f^{(2)}$, $f^{(3)}$, $f^{(4)}$ $U$, $V$, and $W$. The categories $E_8^{(n,\pm)}$ all have the same fusion rules which can be found in \cite[Section 3.3]{MR1145672}, under the translation $f^{(0)} \to \text{id}$, $f^{(1)} \to \rho_1$, $f^{(2)} \to \rho_2$,  $f^{(3)} \to \rho_3$, $f^{(4)} \to \rho_4$, $U \to \rho_7$, $V \to  \rho_5$, and $W \to  \rho_6$.

\vspace{1em}
\end{trivlist}

We have the following folklore result regarding the fusion categories of $ADE$ type, showing that they essentially classify pivotal fusion categories generated by a symmetrical self-dual object of Frobenius-Perron dimension less than 2.
\begin{theorem}\label{thm:ADET}
Let $\cC$ be a pivotal fusion category $\otimes$-generated by a symmetrically self-dual object of Frobenius-Perron dimension less than 2. Then $\cC$ is equivalent to a fusion category of $ADE$ type, or to $\ad(A_{2N}^{(n)})$ for $n \in \Z{2N+1}^\times / \{ \pm \}$.
\end{theorem}

A corollary of the main classification Theorem of this paper gives an improvement to this result, allowing us to remove the words pivotal and symmetrically. 

The fusion categories of $ADE$ type are all $\Z{2}$-graded, and thus have the non-trivial adjoint subcategories $\ad(A_N^{(n)}), \ad(D_{2N}^{(n,\pm)}), \ad(E_6^{(n,\pm)}),$ and $\ad(E_8^{(n,\pm)})$. We call these distinguished subcategories, the \textit{categories of adjoint $ADE$ type.} In the $A_N$ case, the adjoint subcategory is generated by the simple objects 
\[ \left\{  f^{(2n)} : 0 \leq n \leq \frac{N-1}{2}\right\},\] 
in the $D_{2N}$ case the adjoint subcategory is generated by the simple objects 
\[ \{  f^{(2n)} : 0 \leq n \leq N-2\} \cup \{ P, Q\},\]  
in the $E_6$ the adjoint subcategory is generated by the simple objects 
\[ \{ f^{(0)}, f^{(2)}, Z\},\] 
and in the $E_8$ case the adjoint subcategory is generated by the simple objects
 \[\{ f^{(0)}, f^{(2)}, f^{(4)}, W   \}.\]

We have the following monoidal equivalences between the fusion categories of adjoint $ADE$ type.
\begin{align*}
\ad(A_N^{(n)}) &\simeq \ad(A_N^{(-n)}), \\
\ad(D_{2N}^{(n,+)}) \simeq \ad(D_{2N}^{(-n,+)}) &\simeq \ad(D_{2N}^{(n,-)}) \simeq \ad(D_{2N}^{(-n,-)}),\\
\ad(E_6^{(n,+)}) &\simeq \ad(E_6^{(-n,+)}), \\ 
\ad(E_6^{(n,-)}) &\simeq \ad(E_6^{(-n,-)}), \\ 
\ad(E_8^{(n,+)}) &\simeq \ad(E_8^{(-n,+)}), \\ 
\ad(E_8^{(n,-)}) &\simeq \ad(E_8^{(-n,-)}). \\ 
\end{align*}
Thus categories of adjoint $A_N$ type are classified by $n\in \Z{N+1}^\times / \{\pm\}$, categories of adjoint $D_{2N}$ type are classified by $n\in \Z{4N-3}^\times / \{\pm\}$, categories of adjoint $E_6$ type are classified by $n\in \Z{12}^\times / \{\pm\}$ along with a choice of sign, and categories of adjoint $E_8$ type are classified by $n\in \Z{30}^\times / \{\pm\}$ along with a choice of sign. As $\ad(D_{2N}^{(n,+)}) \simeq \ad(D_{2N}^{(n,-)})$ we simply write (in a slight abuse of notation) $\ad(D_{2N}^{(n)})$ for the categories of adjoint $D_{2N}$ type.

The following Proposition shows that if a pivotal fusion category has the same fusion rules as a category of adjoint $ADE$ type, then it is actually equivalent to a category of adjoint $ADE$ type.
\begin{prop}\label{lem:adclass}
Let $\cC$ be a pivotal fusion category with the same fusion rules as either $\ad(A_N^{(n)})$, $\ad(D_{2N}^{(n)})$, $\ad(E_6^{(n,\pm)})$, or $\ad(E_8^{(n,\pm)})$, then $\cC$ is monoidally equivalent to one of the categories
\[  \ad(A_N^{(n)}) ,\quad  \ad(D_{2N}^{(n)}),\quad  \ad(E_{6}^{(n,\pm)}),\quad \text{ or } \quad \ad(E_8^{(n,\pm)})\]
respectively.
\end{prop}
\begin{proof}
The adjoint $A_N$ case is exactly \cite[Theorem A.3]{1801.04409}.

Let $\cC$ be a pivotal fusion category with adjoint $D_{2N}$ fusion rules, and let $X$ be the "$f^{(2)}$" object of $\cC$. As $\operatorname{dimHom}(X^{\otimes p} \to \mathbf{1})$ is equal to $(1,0,1,1,3)$ for $0 \leq p \leq 4$, we have from \cite{MR3624901} that the category generated by $X$ and the trivalent vertex in $\Hom(X^{\otimes 3}\to \mathbf{1})$ is equivalent to $\ad(A_{4N-3}^{(n)})$ for some $n \in \Z{4N-2}^\times$. Thus there exists a dominant functor
\[  \ad(A_N^{(n)}) \to \cC. \]
Via \cite{MR2863377} we have that $\cC$ is equivalent to the fusion category of $(A,\sigma)$ modules in $\ad(A_N^{(n)})$, where $(A,\sigma)$ is a simple central commutative algebra in $\ad(A_N^{(n)})$. Considering global Frobenius-Perron dimensions shows that the Frobenius-Perron dimension of $A$ is two, and thus $A = f^{(0)} \oplus f^{(4N-4)}$. Therefore $\cC$ is a de-equivariantization of $\ad(A_N^{(n)})$ by the subcategory $\langle f^{(4N-4)} \rangle$. Thus
\[  \cC \simeq \ad\left({A_N^{(n)}}\right)_{\langle f^{(4N-4)} \rangle} \simeq \ad\left({A_N^{(n)}}_{\langle f^{(4N-4)} \rangle}\right) \simeq \ad\left(D_{2N}^{(n,\pm)}\right) =  \ad\left(D_{2N}^{(n)}\right).\]

The adjoint $E_6$ case is exactly \cite{1309.4822}.

The adjoint $E_8$ case follows near identically to the adjoint $D_{2N}$ case. The same argument shows that if $\cC$ has adjoint $E_8$ fusion rules, then $\cC$ must be the category of $(f^{(0)} \oplus f^{(10)}\oplus f^{(18)}\oplus f^{(28)},\sigma)$ modules in $\ad(A_{29}^{(n)})$. Hence
\[  \cC \simeq \operatorname{Mod}_{\ad(A_{29}^{(n)})}( f^{(0)} \oplus f^{(10)}\oplus f^{(18)}\oplus f^{(28)},\sigma   ) \simeq  \ad( \operatorname{Mod}_{A_{29}^{(n)}}( f^{(0)} \oplus f^{(10)}\oplus f^{(18)}\oplus f^{(28)},\sigma   )) \simeq \ad( E_8^{(n,\pm)}).\]
\end{proof}

The key tool in proving the main classification result of this paper is the following Theorem, due to Morrison and Snyder. 

\begin{theorem}\label{thm:ADEext}
Let $\cC$ be a fusion category $\otimes$-generated by an object $X$, with Frobenius-Perron dimension less than 2, such that $\ad(\cC) = \langle X\otimes X^*\rangle$. Then $\cC$ is a cyclic extension of a category of adjoint $ADE$ type.
\end{theorem}
\begin{proof}
et $\cC$ be a fusion category $\otimes$-generated by an object $X$, with Frobenius-Perron dimension less than 2, such that $\ad(\cC) = \langle X\otimes X^*\rangle$. Choose
\[  \raisebox{-.5\height}{ \includegraphics[scale = .5]{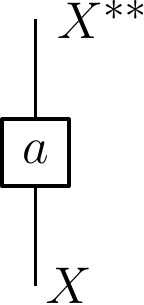}}: X\to X^{**}\]
an arbitrary choice of isomorphism, which exists by \cite[Proposition 2.1]{MR2183279}. Define $\delta_a^\pm$ the scalers (depending on choice of isomorphism $a$) by:
\[ \delta_a^+:= \raisebox{-.5\height}{ \includegraphics[scale = .5]{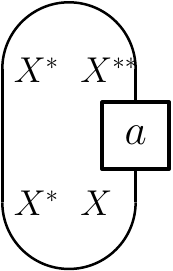}}, \qquad \text{and} \qquad \delta_a^-:= \raisebox{-.5\height}{ \includegraphics[scale = .5]{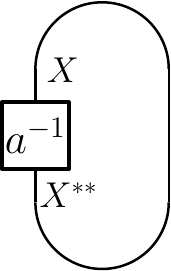}}.  \]
Taking the product of $\delta_a^+$ and $\delta_a^-$ gives a scaler 
\[    |X|^2 := \delta_a^+ \delta_a^-\]
which is independent of choice of isomorphism $a$. The scaler $|X|^2$ is typically called the Muger squared dimension of $X$.

Consider the self-dual object $B_X := X\otimes X^* - \mathbf{1}$, and $d^{(2)} \in \Hom( X\otimes X^* \to X\otimes X^*)$ the idempotent projecting onto $B_X$, given by:
\[   \raisebox{-.5\height}{ \includegraphics[scale = .5]{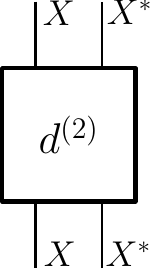}} := \raisebox{-.5\height}{ \includegraphics[scale = .6]{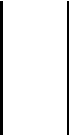}} - \frac{1}{\delta_a^{-} } \raisebox{-.5\height}{ \includegraphics[scale = .6]{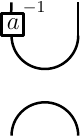}}.\]

We have the following morphisms:
\[ \operatorname{ev}_{X,a} :=  \raisebox{-.5\height}{ \includegraphics[scale = .5]{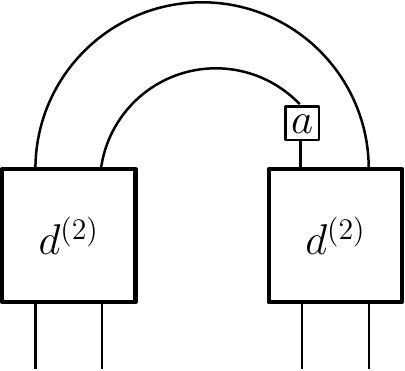}} : B_X \otimes B_X \to \mathbf{1} ,\]
\[ \operatorname{coev}_{X,a} :=  \raisebox{-.5\height}{ \includegraphics[scale = .5]{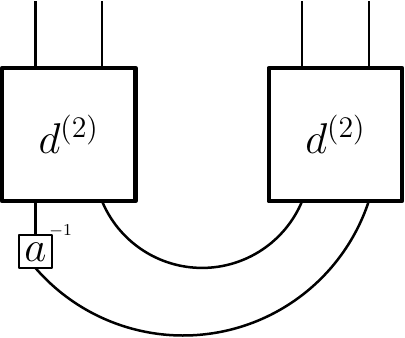}} : \mathbf{1}\to B_X\otimes B_X ,\]
\[ m_{X,a} :=\sqrt{\frac{\delta_a^-}{\delta_a^+\delta_a^- - 2}}  \raisebox{-.5\height}{ \includegraphics[scale = .5]{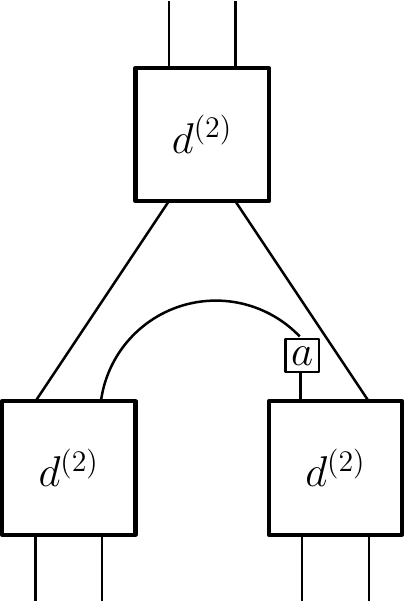}} : B_X\otimes B_X \to B_X ,\]
and
\[ w_{X,a} := \sqrt{\frac{\delta_a^-}{\delta_a^+\delta_a^- - 2}} \raisebox{-.5\height}{ \includegraphics[scale = .5]{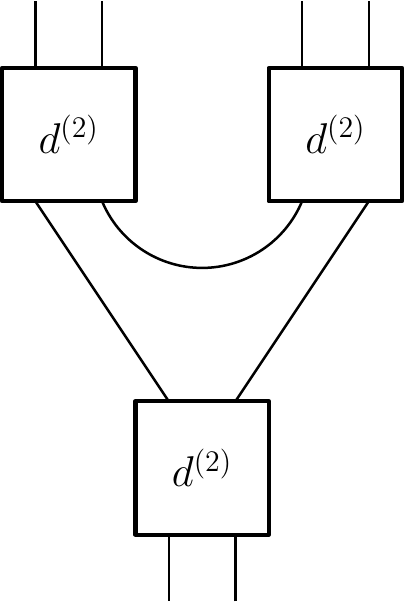}} :  B_X \to B_X\otimes B_X ,\]

A direct computation verifies that these four morphisms satisfy the standard $\operatorname{SO}(3)_q$ relations with loop parameter $|X|^2-1$. To determine $q$ we observe
\[  |X|^2 - 1 = [3]_q = [2]_q^2 - 1\]
thus,
\[ |X|^2 = [2]^2_q.\]
As the Frobenius-Perron dimension of $X$ is less than 2, we have 
\[  | |X|^2 | < 4,\]
and thus
\[ -2 < [2]_q < 2.\]
As $\cC$ is semi-simple, we thus have $q$ is a root of unity. 

As the four morphisms $\operatorname{ev}_{X,a}$, $\operatorname{coev}_{X,a}$, $m_{X,a}$, and $w_{X,a}$ satisfy the $\operatorname{SO}(3)_q$ relations for $q$ a root of unity, we get a dominant tensor functor from the semi-simplification of $\operatorname{SO}(3)_q$ to the category $\otimes$-generated by $B_X$. Hence by \cite{MR2863377} we have that $\langle B_X \rangle$ is monoidally equivalent to the category of $(A,\sigma)$ modules in $\operatorname{SO}(3)_q$, for $(A,\sigma)$ a central commutative algebra object. As $q$ is a root of unity, we have that 
\[  \ad( \overline{\operatorname{SU}(2)_q}) \simeq \overline{\operatorname{SO}(3)_q}.\]
Thus $\langle B_X \rangle$ is equivalent to the adjoint subcategory of $(A,\sigma)$ modules in $\operatorname{SU}(2)_q$, and hence equivalent to a category of adjoint $ADE$ type.

By our assumption that $\ad(\cC) = \langle X\otimes X^*\rangle$, we thus have that $\cC$ is a $G$-graded extension of a category of adjoint $ADE$ type. Further, as $\cC$ is $\otimes$-generated by a single object, we have that this $G$-graded extension is in fact a cyclic extension.

\end{proof}

Coupled with the results of \cite{MR2677836}, allowing us to classify $\Z{M}$-graded extensions, the above Theorem provides a practical base to prove the main classification result of this paper. We recall the relevant information about the categories of adjoint $ADE$ type, necessary to apply the classification of graded extensions to these categories. This information was computed in \cite{MR3808052}.

We begin by giving explicit descriptions of the invertible bimodules over the categories of adjoint $ADE$ type, along with the order of each bimodule. This information will be useful when we try to classify cyclic homomorphisms into the Brauer-Picard groups.

Let $\cM$ be an invertible bimodule over a $\Z{2}$-graded fusion category $\cC$. Then $\cM$ splits into two invertible bimodule categories over $\ad(\cC)$. We call these two $\ad(\cC)$ bimdoules, $\cM^\text{even}$ and $\cM^\text{odd}$ respectively. We can realise all the bimodules over the categories of adjoint $ADE$ type as the even and odd parts of certain bimodules over the categories of $ADE$ type, along with twistings by monoidal auto-equivalences as in Definition~\ref{dfn:twist}.

Over the $A_N^{(n)}$ fusion categories we have the trivial $A_N$ bimodule for all $N$, and the $D_{\frac{N+3}{2}}$ bimodule when $N \equiv 3 \pmod 4$ (we also have the $D_\text{even}$ modules when $N \equiv 1 \pmod 4$, but these don't have the structure of invertible bimodules over $A_N^{(n)}$). The auto-equivalences of $\ad(A_N^{(n)})$ are trivial, except for when $N = 7$, in which case there is a single non-trivial auto-equivalence sending $f^{(2)} \leftrightarrow f^{(4)}$.

 Over the $D_{2N}^{(n,\pm)}$ fusion categories we have the trivial $D_{2N}$ bimodule for all $N$, and the $E_7$, and $\overline{E_7}$ bimodules when $N=5$ (both $E_7$ and $\overline{E_7}$ have the same bimodule fusion rules). There is always an order two auto-equivalence of $\ad(D_{2N}^{(n)})$ that sends $P \leftrightarrow Q$, and when $N=5$ there is also an order 3 auto-equivalence sending $f^{(2)} \mapsto P \mapsto Q \mapsto f^{(2)}$. 

Over the fusion categories $E_6^{(n,\pm)}$ and $E_8^{(n,\pm)}$ there is just the trivial bimodule, and no non-trivial auto-equivalences of the adjoint subcategories. We summarise this information in Table~\ref{table:bim}.

\begin{sidewaystable}
\small
   
\centering 
\vspace{16cm}
    \begin{tabular}{c | l | c | l}
    	\toprule
			\multicolumn{2}{l |}{$\cC $}      & Bimodule categories over $\cC$ & Orders     \\
	\midrule
	            $\ad(A_N^{(n)})$   &  $N=3$ & $A_3^\text{even}$ and $A_3^\text{odd}$  &  $1,2$ \\							    		     		   
		 &  $N=7$ & $A_7^\text{even}$, $A_7^\text{odd}$,$D_5^\text{even}$, $D_5^\text{odd}$,$_{f^{(2)} \leftrightarrow f^{(4)} }A_7^\text{even}$, $_{f^{(2)} \leftrightarrow f^{(4)} }A_7^\text{odd}$,$_{f^{(2)} \leftrightarrow f^{(4)} }D_5^\text{even}$, and $_{f^{(2)} \leftrightarrow f^{(4)} }D_5^\text{odd}$ & $1,2,2,2,2,4,4,2$ \\		
                       & $N \equiv 0 \pmod 2$ &  $A_{N}^\text{even}$ & $1$ \\
		  & $N \equiv 1 \pmod 4$ &  $A_{N}^\text{even}$ and $A_{N}^\text{odd}$ & $1,2$ \\
		& $N \equiv 3 \pmod 4$ and $N \neq \{3,7\}$  & $A_{N}^\text{even}$, $A_{N}^\text{odd}$, $D_{\frac{N+1}{2}+1}^\text{even}$, and $D_{\frac{N+1}{2}+1}^\text{odd}$ & $1,2,2,2$\\

\midrule

$\ad(D_{2N}^{(n)})$ & $N \neq 5$ & $D_{2N}^\text{even}$, $D_{2N}^\text{odd}$,$_{P\leftrightarrow Q}D_{2N}^\text{even}$, and $_{P\leftrightarrow Q}D_{2N}^\text{odd}$ & $1,2,2,2$\\

                      & $N=5$ & $D_{10}^\text{even}$, $D_{10}^\text{odd}$, $E_7^\text{even}$, $\overline{E_7}^\text{even}$, $E_7^\text{odd}$, $\overline{E_7}^\text{odd}$,      & $1,2,2,2,3,3$ \\
                      &		    &$_{P\leftrightarrow Q}D_{10}^\text{odd}$, $_{P\leftrightarrow Q}D_{10}^\text{even}$,  $_{P\leftrightarrow Q}E_7^\text{odd}$, $_{P\leftrightarrow Q}\overline{E_7}^\text{odd}$, $_{P\leftrightarrow Q}\overline{E_7}^\text{even}$, $_{P\leftrightarrow Q}E_7^\text{even}$, & $2,2,2,2,6,6$\\
                      &		    & $_{f^{(2)} \leftrightarrow P}E_7^\text{even}$,$_{f^{(2)} \leftrightarrow P}E_7^\text{odd}$, $_{f^{(2)} \leftrightarrow P}D_{10}^\text{even}$,  $_{f^{(2)} \leftrightarrow P}\overline{E_7}^\text{odd}$,  $_{f^{(2)} \leftrightarrow P}D_{10}^\text{odd}$,$_{f^{(2)} \leftrightarrow P}\overline{E_7}^\text{even}$, & $2,2,2,2,6,6$\\
                      &		    &  $_{f^{(2)} \leftrightarrow Q}\overline{E_7}^\text{even}$, $_{f^{(2)} \leftrightarrow Q}E_7^\text{odd}$,$_{f^{(2)} \leftrightarrow Q}\overline{E_7}^\text{odd}$, $_{f^{(2)} \leftrightarrow Q}D_{10}^\text{even}$,  $_{f^{(2)} \leftrightarrow Q}E_7^\text{even}$,$_{f^{(2)} \leftrightarrow Q}D_{10}^\text{odd}$,& $2,2,2,2,6,6$\\
                      &		    & $_{f^{(2)} \mapsto P \mapsto Q}E_7^\text{odd}$, $_{f^{(2)} \mapsto P \mapsto Q}\overline{E_7}^\text{even}$,$_{f^{(2)} \mapsto P \mapsto Q}D_{10}^\text{odd}$,$_{f^{(2)} \mapsto P \mapsto Q}E_7^\text{even}$,  $_{f^{(2)} \mapsto P \mapsto Q}D_{10}^\text{even}$,$_{f^{(2)} \mapsto P \mapsto Q}\overline{E_7}^\text{odd}$,  & $3,6,6,6,3,3$\\
		     &		    & $_{f^{(2)} \mapsto Q \mapsto P}\overline{E_7}^\text{odd}$, $_{f^{(2)} \mapsto Q \mapsto P}E_7^\text{even}$, $_{f^{(2)} \mapsto Q \mapsto P}\overline{E_7}^\text{even}$,$_{f^{(2)} \mapsto Q \mapsto P}D_{10}^\text{odd}$, $_{f^{(2)} \mapsto Q \mapsto P}E_7^\text{odd}$, $_{f^{(2)} \mapsto Q \mapsto P}D_{10}^\text{even}$ & $3,6,6,6,3,3$\\
\midrule

$\ad(E_6^{(n,\pm)})$ & & $E_6^\text{even}$ and $E_6^\text{odd}$ & $1,2$\\

\midrule

$\ad(E_8^{(n,\pm)})$ & & $E_8^\text{even}$ and $E_8^\text{odd}$ & $1,2$\\
				                 
    	\bottomrule
    \end{tabular} 
  \caption{\label{table:bim} Bimodules over the categories of adjoint $ADE$ type}
  
\end{sidewaystable}

We also present Table~\ref{tab:Inv} showing the invertible objects in the centres of each category of adjoint $ADE$ type. Details on the notations used for the objects in the centres can be found in \cite{MR3808052}. When an invertible bimodule acts non-trivially on the group of invertible objects in the centre, we also include the information of the action. 
\begin{sidewaystable}

\centering 
\vspace{16cm}
    \begin{tabular}{c | l | c | l}
    	\toprule
			\multicolumn{2}{l |}{$\cC $}      & $\Inv(\cZ(\cC))$ & Action of bimodules (when non-trivial)     \\
	\midrule
	            $\ad(A_N^{(n)})$   &  $N=3$ & $\{\mathbf{1}\boxtimes\mathbf{1}, \frac{ f^{(1)}\boxtimes f^{(1)} +S}{2}, \frac{f^{(1)}\boxtimes f^{(1)} -S}{2},f^{(2)}\boxtimes \mathbf{1} \} \cong \Z{2}\times \Z{2}$  &  $A_3^\text{odd}$ exchanges the objects $ \frac{f^{(1)}\boxtimes f^{(1)} +S}{2}$ and $\frac{f^{(1)}\boxtimes f^{(1)} -S}{2}$  \\	
	            &  $N \equiv 0 \pmod 2$ & $\{\mathbf{1}\boxtimes\mathbf{1} \}$&    \\	
	            &  $N \equiv 1 \pmod 2$ & $\{\mathbf{1}\boxtimes\mathbf{1},f^{(N-1)}\boxtimes\mathbf{1}  \}$&    \\	
	            \midrule
	            $\ad(D_{2N}^{(n,\pm)})$ & $N=2$ &$ \{\mathbf{1}\boxtimes\mathbf{1},\mathbf{1}\boxtimes P,\mathbf{1}\boxtimes Q,$ &  $D_4^\text{odd}$ applies $P\leftrightarrow Q$ to the second factor \\ 		
	                         		   &		&	$P\boxtimes\mathbf{1},P\boxtimes P,P\boxtimes Q,$	  &  $_{P\leftrightarrow Q}D_4^\text{odd}$ applies $P\leftrightarrow Q$ to the first factor \\	
			   	           &		&	\qquad\qquad$ Q\boxtimes\mathbf{1},Q\boxtimes P,Q\boxtimes Q\} \cong \Z{3}\times \Z{3}$    &  $_{P\leftrightarrow Q}D_4^\text{even}$ applies $P\leftrightarrow Q$ to both factors \\
				           & $N>2$ & $\{\mathbf{1}\boxtimes\mathbf{1} \}$ & \\
		\midrule
		$\ad(E_6^{(n,\pm)})$          &		&  $\{	   \mathbf{1}\boxtimes\mathbf{1},f^{(10)}\boxtimes\mathbf{1}    	 \}$    & \\
		\midrule
		$\ad(E_8^{(n,\pm)})$          &		& $\{       \mathbf{1}\boxtimes\mathbf{1}   \}$ &\\  
    	\bottomrule
    \end{tabular}
\caption{\label{tab:Inv} Invertible objects in the centre, and action by invertible bimodules}

\end{sidewaystable}
\subsection{Frobenius-Perron dimensions of simple objects in bimodule categories}\label{sub:dims}

Recall we only care about extensions of $\otimes$-generated by an object of Frobenius-Perron dimension less than 2. We compute the Frobenius-Perron dimensions of the objects in each of our bimodule categories, as this will allow us to rule out many extensions that can not be $\otimes$-generated by such an object, and thus disqualify certain cyclic homomorphisms into the Brauer-Picard group. As twisting a bimodule by a monoidal auto-equivalence of the underlying category doesn't change the Frobenius-Perron dimensions of the objects, we only include the Frobenius-Perron dimensions of the untwisted bimodules.

\begin{trivlist}\leftskip=2em
\item\textbf{Dimensions in the $A$ series:}

Let $q = e^\frac{\pi i}{N+1}$, then the Frobenius-Perron dimensions of the simple objects in the invertible bimodules (when they exist) over $\ad(A_N^{(n)})$ are:

\begin{center}
\begin{tabular}{c | c}
    	\toprule
			Bimodule   & Dimensions of simples     \\
	\midrule
	$A_N^\text{even}$ & $\{ [2n-1]_q : 1 < n < \lceil\frac{N}{2}\rceil \}$ \\
	 $A_N^\text{odd}$ & $\{ [2n]_q : 1 < n < \lfloor\frac{N}{2}\rfloor \} $ \\
	 $D_{\frac{N+1}{2}+1}^\text{even} $&  $\{ \sqrt{2} [2n-1]_q : 1 < n \leq \frac{N+1}{4}  \} $ \\
	 $D_{\frac{N+1}{2}+1}^\text{odd}$ & $\{ \sqrt{2} [2n]_q : 1 < n < \frac{N+1}{4}  \}\bigcup \{ \sqrt{2}[\frac{N+1}{2}]_q \} $.           
    \end{tabular}
\end{center}
\vspace{1em}
\item\textbf{Dimensions in the $D$ series:}

Let $q = e^\frac{\pi i}{4N-2}$, then the Frobenius-Perron dimensions of the simple objects in the invertible bimodules over $\ad(D_{2N}^{(n)})$ (when they exist) are approximately:
\begin{center}
\begin{tabular}{c | c}
    	\toprule
			Bimodule   &  Dimensions of simples     \\
	\midrule
	$D_{2N}^\text{even}$ & $ \{ [2n-1]_q : 1 < n < N \} \bigcup \{ \frac{[2N-1]_q}{2}\} $ \\
	$D_{2N}^\text{odd}$ & $\{ [2n]_q : 1 < n < N \}$\\
	$E_7^\text{even}$ &  $\{1.96962, 3.70167, 4.98724, 5.67128\}$ \\
	$E_7^\text{odd}$ & $\{2.53209,3.87939, 7.29086\}$.
    \end{tabular}
\end{center}
\vspace{1em}
\item\textbf{Dimensions in the $E$ series:}

The Frobenius-Perron dimensions of the simple objects in the invertible bimodules over $\ad(E_6^{(n,\pm)})$ are approximately:
\begin{center}
\begin{tabular}{c | c}
    	\toprule
			Bimodule   & Dimensions of simples     \\
	\midrule
	$E_6^\text{even}$ & $\{1, 2.73205\}$ \\
	$E_6^\text{odd}$ & $\{ 1.93185 \}$.
    \end{tabular}
\end{center}

The Frobenius-dimensions of the simple objects in the invertible bimodules over $\ad(E_8^{(n,\pm)})$ are approximately:
\begin{center}
\begin{tabular}{c | c}
    	\toprule
			Bimodule  & Dimensions of simples     \\
	\midrule
	$E_8^\text{even}$ & $\{1,1.61803,2.9563,4.78339\}$ \\
	$E_8^\text{odd}$ & $\{1.98904,2.40487,3.21834,3.89116\}$.
    \end{tabular}
\end{center}
\end{trivlist}

\section{Classification of Cyclic Extensions}\label{sec:class}
We are now in place to begin classifying fusion categories as in the statement of Theorem~\ref{thm:main}. By Theorem~\ref{thm:ADEext} such categories must be cyclic extensions of a category of adjoint $ADE$ type. Thus in this section we compute cyclic extensions of the adjoint subcategories of the $ADE$ fusion categories, that are $\otimes$-generated by an object of Frobenius-Perron dimension less than 2. 

Our proofs all follow the same outline. First we begin by classifying cyclic homomorphisms into the Brauer-Picard group of each category. As we only care about extensions that are $\otimes$-generated by an object of Frobenius-Perron dimension less than 2, we are able to rule out many of these cyclic homomorphisms simply by considering the Frobenius-Perron dimensions of the objects in the invertible bimodules.

Next we use the classification theory of graded extensions to count an upper bound for the number of possible extensions corresponding to each homomorphism. As we only care about extensions up to twisting the associator, we have from the extension theory that an upper bound is given by the cohomology group
\[  H^2(\Z{M}, \Inv(\cZ(\cC))),\]
where $\Z{M}$ acts on the group $\Inv(\cZ(\cC))$ by the specified homomorphism $\Z{M} \to \BrPic(\cC)$. Computing this group in each case is a straightforward exercise in group cohomology. 

We then construct extensions of $\cC$ to realise the upper bound. We are extremely lucky in that each of these upper bounds turns out to be sharp. For the most part these constructions are straightforward, just involving Deligne products and de-equivariantizations of well known categories. However in the adjoint $A_7$ and adjoint $D_{10}$ cases we find interesting extensions realised as quantum subgroups of certain quantum group categories.

In several cases we run into the problem that categories which are inequivalent as $\Z{M}$-graded extensions over $\cC$, can be monoidally equivalent when just considered as monoidal categories. This phenomenon only occurs in the adjoint $A_7$, $D_{4}$, and $D_{10}$ cases. For these cases, we apply the techniques developed in \cite{1711.00645} to determine which inequivalent extensions are equivalent as monooidal categories.

We proceed case by case for each category of adjoint $ADE$ type. Instead of the standard lexicographic order, we instead order these cases by difficulty of proof. 

\subsection*{Cyclic extensions of categories of adjoint $A_{2N}$ type}
By far the easiest cases are the categories $\ad(A_{2N}^{(n)})$. This is due to the fact that the Brauer-Picard group is trivial, and there are no non-trivial invertible objects in the centre. Thus we begin our classifications with this case.
\begin{lemma}\label{lem:aeven}
Fix $n\in \Z{2N+1}^\times / \{\pm\}$. If $\cC$ is a $\Z{M}$-graded extension of $\ad(A_{2N}^{(n)})$, $\otimes$-generated by an object of Frobenius-Perron dimension less than 2, then, up to twisting the associator of $\cC$ by an element of $H^3(\Z{M}, \mathbb{C}^\times)$, the category $\cC$ is monoidally equivalent to 
\[
 \ad(A_{2N}^{(n)}) \boxtimes \operatorname{Vec}(\Z{M}).
\]
\end{lemma}
\begin{proof}
We begin by classifying group homomorphisms $\Z{M} \to \BrPic(\ad(A_{2N}^{(n)}))$. As the Brauer-Picard group of $\ad(A_{2N}^{(n)})$ is trivial the only homomorphism $\Z{M} \to \BrPic(\ad(A_{2N}^{(n)}))$ is the map $1 \mapsto A^\text{even}_{2N}$.

From Table~\ref{tab:Inv} we know that $\cZ(\ad(A_{2N}^{(n)}))$ has no invertible objects. Therefore the group
\[ H^2(\Z{M},\Inv(\cZ(\ad(A_{2N}))))\] must be trivial for all $M$. Hence, up to twisting the associator by an element of $H^3(\Z{M}, \mathbb{C}^\times)$, there is a unique $\Z{M}$-graded extension of $\ad(A_{2N}^{(n)})$. This extension is realised by the category $ \ad(A_{2N}^{(n)}) \boxtimes \operatorname{Vec}(\Z{M})$.
\end{proof}

\subsection*{Cyclic extensions of the categories of adjoint $E_8$ type}
The adjoint $E_8$ case is just slightly more complicated that the adjoint $A_{2N}$ case. While there are no non-trivial invertible elements in the centre of $\ad(E_8^{(n,+)})$, we have now that the Brauer-Picard group is non-trivial. However only one of the bimodule categories over $\ad(E_8^{(n,+)})$ contains an object of the correct dimension to $\otimes$-generate a cyclic extension of $\ad(E_8^{(n,+)})$. Thus, up to twisting the associator, there is a unique $\Z{M}$-graded extension of $\ad(E_8^{(n,+)})$, $\otimes$-generated by an object of Frobenius-Perron dimension less than 2. This extension is easy to construct, so we can jump straight into the classification result of this subsection.

\begin{lemma}\label{lem:e8}
Fix $n\in \Z{30}^\times / \{\pm \}$. If $\cC$ is a $\Z{M}$-graded extension of $\ad(E_8^{(n,\pm)})$, $\otimes$-generated by an object of Frobenius-Perron dimension less than 2, then $M$ is even, and, up to twisting the associator of $\cC$ by an element of $H^3(\Z{M}, \mathbb{C}^\times)$, the category $\cC$ is monoidally equivalent to: 
\[   \subcat{E_8^{(n,\pm)} \boxtimes \Vec(\Z{M})}.\]
\end{lemma}
\begin{proof}
Recall that the Brauer-Picard group of $\ad(E_8^{(n,\pm)})$ is $\Z{2}$. The only object in the trivial bimodule $E_8^\text{even}$ with dimension less than $2$ has dimension $\frac{1+\sqrt{5}}{2}$. Therefore the category generated by this object would be a cyclic extension of a category of adjoint $A_4$ type, and could not generate all of $\cC$, as $\cC$ contains an $\ad(E_8^{(n,\pm)})$ subcategory. Thus we can assume $M$ is even and the homomorphism $\Z{M} \to \BrPic(\ad(E_8^{(n,\pm)}))$ is defined by $1 \mapsto E_8^\text{odd}$. From Table~\ref{tab:Inv} there are no non-trivial invertible elements in the centre of $\ad(E_8^{(n,\pm)})$. Thus 
\[ H^2(\Z{M}, \Inv(\cZ(\ad(E_8^{(n,\pm)})))) = \{e\},\] 
and so, up to twisting the associator by an element of $H^3(\Z{M}, \mathbb{C}^\times)$, there is a unique $\Z{M}$-graded extension of $\ad(E_8^{(n,\pm)})$ $\otimes$-generated by an object of Frobenius-Perron dimension less than 2. This extension is realised by the category
\[   \subcat{E_8^{(n,\pm)} \boxtimes \Vec(\Z{M})}.\]
\end{proof}

\subsection*{Cyclic extensions of categories of adjoint $D_{2N}$ type, $N \neq \{2,5\}$}
Assuming that $N \notin \{2,5\}$, the categories $\ad(D_{2N}^{(n)})$ have no non-trivial invertible objects in the centre, and Brauer-Picard group isomorphic to $\Z{2}\times \Z{2}$. Only two of these bimodules contain an object that could tensor generate a cyclic extension of $\ad(D_{2N}^{(n)})$, both of which have order 2. Thus there are no $\Z{2M+1}$-graded extensions of $\ad(D_{2N}^{(n)})$, $\otimes$-generated by an object of Frobenius-Perron dimension less than 2, and, up to twisting the associator, at most two $\Z{2M}$-graded extensions of $\ad(D_{2N}^{(n)})$, $\otimes$-generated by an object of Frobenius-Perron dimension less than 2. These two extensions exist, and are realised by the categories
\[  \subcat{  D_{2N}^{(n,+)} \boxtimes \Vec(\Z{2M})} \text{ and }   \subcat{  D_{2N}^{(n,-)} \boxtimes \Vec(\Z{2M})}.  \]
While both these categories are clearly $\Z{2M}$-graded extensions of $\ad(D_{2N}^{(n)})$, the difficulty in this case comes in showing that these two categories are non-equivalent, even up to twisting the associator by an element of $H^3(\Z{2M},\mathbb{C}^\times)$.

\begin{lemma}\label{cor:d2ndiff}
Fix $n\in \Z{4N-2}^\times / \{\pm \}$. The categories \[    D_{2N}^{(n,+)} \text{ and }  D_{2N}^{(n,-)}  \] are monoidally non-equivalent, even up to twisting the associator by an element of $H^3(\Z{2},\mathbb{C}^\times)$.
\end{lemma}
\begin{proof}
Twisting the associator of $D_{2N}^{(n,+)}$ by the non-trivial element of $H^3( \Z{2} , \mathbb{C}^\times)$ gives the category $D_{2N}^{(-n,+)}$. Thus, as the categories $D_{2N}^{(-n,+)}$ and $D_{2N}^{(n,-)}$ are monoidally non-equivalent, the categories $D_{2N}^{(n,+)}$ and $D_{2N}^{(n,-)}$ are monoidally non-equivalent, even up to twisting the associator.
\end{proof}

\begin{cor}\label{lem:d2ndiff}
Fix $n\in \Z{4N-2}^\times / \{\pm \}$. The categories 
\[  \subcat{  D_{2N}^{(n,+)} \boxtimes \Vec(\Z{2M})} \text{ and }   \subcat{  D_{2N}^{(n,-)} \boxtimes \Vec(\Z{2M})},  \]
 are monoidally non-equivalent, even up to twisting the associator by an element of $H^3(\Z{2M},\mathbb{C}^\times)$.
\end{cor}
\begin{proof}
There exist monoidal embeddings
\[ D_{2N}^{(n,\pm)}\to  \subcat{  D_{2N}^{(n,\pm)} \boxtimes \Vec(\Z{2M})}  \]
sending
\[     f^{(2m)} \to f^{(2m)}\boxtimes 0, \quad   f^{(2m+1)} \to f^{(2m+1)}\boxtimes m, \quad P \to P\boxtimes 0, \text{ and } \quad Q \to Q\boxtimes 0.\]

These monoidal embeddings preserve the graded structure. Hence, if the categories 
\[  \subcat{  D_{2N}^{(n,+)} \boxtimes \Vec(\Z{2M})} \text{ and }   \subcat{  D_{2N}^{(n,-)} \boxtimes \Vec(\Z{2M})} \]
were equivalent, up to twisting the associator, then the subcategories 
\[D_{2N}^{(n,+)} \text{ and } D_{2N}^{(n,-)} \]
would be equivalent, up to twisting the associator. However this is impossible by Lemma~\ref{cor:d2ndiff}.
\end{proof}

With Corollary~\ref{lem:d2ndiff} in hand, we can now prove the classification result for this subsection.

\begin{lemma}\label{lem:d2n}
Let $N \neq \{2,5\}$, and fix $n\in \Z{4N-2}^\times$. If $\cC$ is a $\Z{M}$-graded extension of $\ad(D_{2N}^{(n)})$, $\otimes$-generated by an object of Frobenius-Perron dimension less than 2, then $M$ is even, and, up to twisting the associator of $\cC$ by an element of $H^3(\Z{M}, \mathbb{C}^\times)$, the category $\cC$ is monoidally equivalent to one of: 
\[  \subcat{  D_{2N}^{(n,\pm)} \boxtimes \Vec(\Z{M})}.  \]
\end{lemma}
\begin{proof}
Recall that the Brauer-Picard group of $\ad(D_{2N}^{(n)})$ is $\Z{2}\times \Z{2}$. To determine possible homomorphisms $\Z{M} \to \BrPic(\ad(D_{2N}^{(n)}))$, that could possibly give rise to an extension, $\otimes$-generated by an object of Frobenius-Perron dimension less than 2, we have to split into cases.

\begin{trivlist}\leftskip=2em
\item \textbf{Case $N=3$:}

When $N=3$ the trivial bimodule $D_6^\text{even}$, and the twisted trivial bimodule $ _{P\leftrightarrow Q}D_6^\text{even}$ both contain non-trivial objects of dimension $\frac{1+\sqrt{5}}{2}$. However any category generated by such an object couldn't generate all of $\cC$. This leaves two homomorphisms $\phi: \Z{M} \to \BrPic(\ad(D_6^{(n)}))$ to consider, the map defined by $1 \mapsto D_6^\text{odd}$ and the map defined by $1 \mapsto _{P\leftrightarrow Q}D_6^\text{odd}$. In particular we may conclude that $M$ is even.

\vspace{1em}

\item \textbf{Case $N \neq \{2,3,5\}$:}

For these cases the only bimodules over $\ad(D_{2N}^{(n)})$ with an object of dimension less than $2$ are $D_{2N}^\text{odd}$ and $_{P\leftrightarrow Q}D_{2N}^\text{odd}$. This leaves two homomorphisms $\phi: \Z{M} \to \BrPic(\ad(D_{2N}^{(n,+)}))$ to consider, the map defined by $1 \mapsto D_{2N}^\text{odd}$ and the map defined by $1 \mapsto _{P\leftrightarrow Q}D_{2N}^\text{odd}$. In particular we may conclude that $M$ is even.
\end{trivlist}\vspace{1em}
For either case we see that $M$ must be even, and there are two homomorphisms $\phi: \Z{M} \to \BrPic(\ad(D_{2N}^{(n)}))$ to consider
\[ 1 \mapsto D_{2N}^\text{odd}\text{ and }1 \mapsto _{P\leftrightarrow Q}D_{2N}^\text{odd}.  \] 

When $N>2$ the centre of $\ad(D_{2N}^{(n)})$ contains no non-trivial invertible objects, and hence $H^2(\Z{M}, \Inv(\cZ(\ad(D_{2N})))) = \{e\}$. Thus, up to twisting the associator by an element of $H^3(\Z{M}, \mathbb{C}^\times)$, there are at most two $\Z{M}$-graded extensions of $\ad(D_{2N}^{(n)})$ generated by an object of Frobenius-Perron dimension less than 2. The categories
\[ \subcat{  D_{2N}^{(n,+)} \boxtimes \Vec(\Z{M})},\]
and 
\[ \subcat{  D_{2N}^{(n,-)} \boxtimes \Vec(\Z{M})},\]
are two such extensions, which by Lemma~\ref{lem:d2ndiff} are not equivalent up to twisting the associator. Hence these two categories, up to twisting the associator by an element of $H^3(\Z{M}, \mathbb{C}^\times)$, realise all $\Z{M}$-graded extensions of $\ad(D_{2N}^{(n)})$, generated by an object of Frobenius-Perron dimension less than 2.
\end{proof}

\subsection*{Cyclic extensions of categories of adjoint $A_{2N+1}$ type, $N \neq \{1,3\}$ }

The adjoint $A_\text{odd}$ cases turn out to be particularly difficult. In this subsection we restrict our attention to $N \notin \{1,3\}$. While there is only a single bimodule over the category $\ad(A_{2N+1}^{(n)})$ that could give rise to a $\Z{M}$-graded extension of $\ad(A_{2N+1}^{(n)})$ $\otimes$-generated by an object of Frobenius-Perron dimension less than 2, we now have that there are now non-trivial invertible objects in the centre. These invertible elements form a group isomorphic to $\Z{2}$, and hence we have that 
\[   H^2(\Z{2M} , \Inv(\cZ(\ad(A_{2N+1})))) = \Z{2},\]
regardless of choice of homomorphism $\Z{2M} \to \BrPic(\ad(A_{2N+1}))$. Thus, up to twisting the associator, there exist two possible $\Z{2M}$-graded extensions of $\ad(A_{2N+1}^{(n)})$ $\otimes$-generated by an object of Frobenius-Perron dimension less than 2. We begin this subsection by constructing both of these extensions.

One of these $\Z{2M}$-graded extensions is easy to construct, and is realised by the category
\[ \subcat{A_{2N+1}^{(n)}\boxtimes \LVec( \Z{2M} )  }.   \]
Constructing the other extension is much more involved.

Consider the non-trivial invertible object $f^{(2N)}$ of $A_{2N+1}^{(n)}$. The subcategory generated by this object lifts to either a copy of $\Rep(\Z{2})$, or of $\sVec$ in the centre, depending on if $N$ is even or odd respectively. 

If $N$ is even, then by Lemma~\ref{lem:liftyourself}, the subcategory $\langle f^{(2N)} \boxtimes 2M  \rangle$ of $A_{2N+1}^{(n)} \boxtimes \LVec(\Z{4M})$ lifts to a copy of $\Rep(\Z{2})$ in the centre, and hence the subcategory $\langle f^{(2N)} \boxtimes 2M  \rangle$ of $\subcat{A_{2N+1}^{(n)} \boxtimes \LVec(\Z{4M})}$ also lifts to a copy of $\Rep(\Z{2})$ in the centre. Thus we can de-equivariantize $\subcat{A_{2N+1}^{(n)} \boxtimes \LVec(\Z{4M})}$ by the subcategory $\langle f^{(2N)} \boxtimes 2M  \rangle$ to get the new fusion category 
\[    \subcat{A_{2N+1}^{(n)}\boxtimes \LVec( \Z{4M} )  }_{ \langle f^{(2N)} \boxtimes 2M  \rangle    }.   \]

If $N$ is odd, then by Lemma~\ref{lem:liftyourself}, the subcategory $\langle f^{(2N)} \boxtimes 2M  \rangle$ of $A_{2N+1}^{(n)} \boxtimes \IVec(\Z{4M})$ lifts to a copy of $\Rep(\Z{2})$ in the centre, and hence the subcategory $\langle f^{(2N)} \boxtimes 2M  \rangle$ of $\subcat{A_{2N+1}^{(n)} \boxtimes \IVec(\Z{4M})}$ also lifts to a copy of $\Rep(\Z{2})$ in the centre. Thus we can de-equivariantize $\subcat{A_{2N+1}^{(n)} \boxtimes \IVec(\Z{4M})}$ by the subcategory $\langle f^{(2N)} \boxtimes 2M  \rangle$ to get the new fusion category 
\[    \subcat{A_{2N+1}^{(n)}\boxtimes \IVec( \Z{4M} )  }_{ \langle f^{(2N)} \boxtimes 2M  \rangle    }.   \]

\begin{rmk}
In an abuse of notation we will write $ \subcat{A_{2N+1}^{(n)}\boxtimes \Vec^\pm( \Z{4M} )  }_{ \langle f^{(2N)} \boxtimes 2M  \rangle    }$ to mean
\[ \subcat{A_{2N+1}^{(n)}\boxtimes \LVec( \Z{4M} )  }_{ \langle f^{(2N)} \boxtimes 2M  \rangle    }\]
if $N$ is even, and 
\[ \subcat{A_{2N+1}^{(n)}\boxtimes \IVec( \Z{4M} )  }_{ \langle f^{(2N)} \boxtimes 2M  \rangle    }\]
if $N$ is odd.
\end{rmk}

As the action of $ f^{(2N)} \boxtimes 2M$ is fixed point free in either case, the fusion rules of these de-equivariantizations can easily be determined from the fusion rules of $A_{2N+1}^{(n)}$, using the free module functor
\[  \subcat{A_{2N+1}^{(n)}\boxtimes \Vec^{\pm}( \Z{4M} )  } \to  \subcat{A_{2N+1}^{(n)}\boxtimes \Vec^\pm( \Z{4M} )  }_{ \langle f^{(2N)} \boxtimes 2M  \rangle    }.   \]  
The following Lemmas show that the categories $\subcat{A_{2N+1}^{(n)}\boxtimes \Vec^\pm( \Z{4M} )  }_{ \langle f^{(2N)} \boxtimes 2M  \rangle    }$ is a $\Z{2M}$-graded extension of $\ad(A_{2N+1}^{(n)})$, that is inequivalent, even up to twisting the associator, to the category
\[  \subcat{A_{2N+1}^{(n)}\boxtimes \LVec( \Z{2M} )  } .  \]
\begin{lemma}\label{lem:aOddExt}
Let $N \geq 1$ and $M\geq 1$. The category $\subcat{A_{2N+1}^{(n)}\boxtimes \Vec^\pm( \Z{4M} )  }_{ \langle f^{(2N)} \boxtimes 2M  \rangle    }$ is a $\Z{2M}$ graded extension of $\ad(A_{2N+1}^{(n)})$.
\end{lemma}
\begin{proof}
For $0\leq n \leq N$, consider the objects:
\[    \{ (f^{( 2n )} \boxtimes 0) \oplus ( f^{(2N - 2n)} \boxtimes 2M) \} \]
of $\subcat{A_{2N+1}^{(n)}\boxtimes \Vec^\pm( \Z{4M} )  }_{ \langle f^{(2N)} \boxtimes 2M  \rangle    }$. The free module functor 
\[   \subcat{A_{2N+1}^{(n)}\boxtimes \Vec^\pm( \Z{4M} )  } \to \subcat{A_{2N+1}^{(n)}\boxtimes \Vec^\pm( \Z{4M} )  }_{ \langle f^{(2N)} \boxtimes 2M  \rangle    } , \]
allows us to determine that these $N+1$ objects form a subcategory with adjoint $A_{2N+1}$ fusion rules. We use Proposition~\ref{lem:adclass} to see that this subcategory is equivalent to $\ad(A_{2N+1}^{(m)})$ for some choice of $m\in \Z{2N+2}^\times$. Considering categorical dimensions determines that $m=n$, which completes the proof.
\end{proof}

\begin{lemma}\label{lem:aOddDiff}
Let $N \geq 2$ and $M\geq 1$, or $N=1$ and $M$ even. Then the categories $ \subcat{A_{2N+1}^{(n)}\boxtimes \LVec( \Z{2M} )  }$ and $\subcat{A_{2N+1}^{(n)}\boxtimes \Vec^\pm( \Z{4M} )  }_{ \langle f^{(2N)} \boxtimes 2M  \rangle    }$ have different fusion rules.
\end{lemma}
\begin{proof}
We split this proof into two cases, depending on if $M$ is even or odd.

\begin{trivlist}\leftskip=2em
\item \textbf{Case $M$ odd and $N\neq 1$:}

When $M$ is odd, the category $\subcat{A_{2N+1}^{(n)}\boxtimes \LVec( \Z{2M} )  }$ has exactly two self-dual objects of Frobenius-Perron dimension $2\cos\left(\frac{\pi}{2N+2}\right)$. These are the objects
\[   f^{(1)} \boxtimes M \text{ and } f^{(2N-1)} \boxtimes M. \]

On the other hand, in the category $\subcat{A_{2N+1}^{(n)}\boxtimes \Vec^\pm( \Z{4M} )  }_{ \langle f^{(2N)} \boxtimes 2M  \rangle    }$ there are no self-dual objects of Frobenius-Perron dimension $2\cos\left(\frac{\pi}{2N+2}\right)$. To see this, notice that such objects would have to live in the $0$ or $M$ graded pieces of $\subcat{A_{2N+1}^{(n)}\boxtimes \Vec^\pm( \Z{4M} )  }_{ \langle f^{(2N)} \boxtimes 2M  \rangle    }$. Then, simply by considering dimensions, we see that the only possibilities for such objects are:
\[     (f^{(1)}\boxtimes M) \oplus (f^{(2N-1)}\boxtimes 3M) \text{ or }  (f^{(2N-1)}\boxtimes M) \oplus (f^{(1)}\boxtimes 3M).\]
However a direct calculations shows that these two objects are dual to each other. Therefore the fusion rules for the categories $ \subcat{A_{2N+1}^{(n)}\boxtimes \LVec( \Z{2M} )  }$ and $\subcat{A_{2N+1}^{(n)}\boxtimes \Vec^\pm( \Z{4M} )  }_{ \langle f^{(2N)} \boxtimes 2M  \rangle    }$ are different.

\vspace{1em}

\item \textbf{Case $M$ even:}

The group of invertible objects of the category $ \subcat{A_{2N+1}^{(n)}\boxtimes \LVec( \Z{2M} )  }$ can easily be seen to be isomorphic to $\Z{2}\boxtimes \Z{M}$.

On the other hand, there are also $2M$ invertible objects of $\subcat{A_{2N+1}^{(n)}\boxtimes \Vec^\pm( \Z{4M} )  }_{ \langle f^{(2N)} \boxtimes 2M  \rangle    }$. One of these is the object $(f^{(0)}\boxtimes  2 ) \oplus (f^{(2N)}\boxtimes  2M+2 )$. The free module functor 
\[   \subcat{A_{2N+1}^{(n)}\boxtimes \Vec^\pm( \Z{4M} )  } \to  \subcat{A_{2N+1}^{(n)}\boxtimes \Vec^\pm( \Z{4M} )  }_{ \langle f^{(2N)} \boxtimes 2M  \rangle    } \]
allows us to see that the object $(f^{(0)}\boxtimes  2 ) \oplus (f^{(2N)}\boxtimes  2M+2 )$ has order $2M$, and thus the group of invertibles of the category $\subcat{A_{2N+1}^{(n)}\boxtimes \Vec^\pm( \Z{4M} )  }_{ \langle f^{(2N)} \boxtimes 2M  \rangle    }$ is isomorphic to $\Z{2M}$. As $M$ is even, the groups $\Z{2}\boxtimes \Z{M}$ and $\Z{2M}$ are non-isomorphic, and thus the fusion rules for the categories $ \subcat{A_{2N+1}^{(n)}\boxtimes \LVec( \Z{2M} )  }$ and $\subcat{A_{2N+1}^{(n)}\boxtimes \Vec^\pm( \Z{4M} )  }_{ \langle f^{(2N)} \boxtimes 2M  \rangle    }$ are different.
\end{trivlist}
\end{proof}

Now that we have shown that the fusion categories \[ \subcat{A_{2N+1}^{(n)}\boxtimes \Vec^\pm( \Z{4M} )  }_{ \langle f^{(2N)} \boxtimes 2M  \rangle    }.  \] realise the possible interesting extensions of $\ad(A_{2N+1}^{(n)})$, we can complete our classification result for this subsection.

\begin{lemma}\label{lem:aodd}
Let $N \notin \{1,3\}$ a natural number, and fix $n\in \Z{2N+2}^\times / \{\pm\}$. If $\cC$ is a $\Z{M}$-graded extension of $\ad(A_{2N+1}^{(n)})$, $\otimes$-generated by an object of Frobenius-Perron dimension less than 2, then $M$ is even, and, up to twisting the associator of $\cC$ by an element of $H^3(\Z{M}, \mathbb{C}^\times)$, the category $\cC$ is monoidally equivalent to one of: 
\begin{align*}
& \subcat{A_{2N+1}^{(n)}\boxtimes \LVec( \Z{M} )  }  \text{ or,} \\
& \subcat{A_{2N+1}^{(n)}\boxtimes \LVec( \Z{2M} )  }_{ \langle f^{(2N)} \boxtimes M  \rangle    },
\end{align*}
if $N$ is even, and one of :
\begin{align*}
& \subcat{A_{2N+1}^{(n)}\boxtimes \LVec( \Z{M} )  }  \text{ or,} \\
& \subcat{A_{2N+1}^{(n)}\boxtimes \IVec( \Z{2M} )  }_{ \langle f^{(2N)} \boxtimes M  \rangle    },
\end{align*}
if $N$ is odd.
\end{lemma}
\begin{proof}
We begin by classifying group homomorphisms $\Z{M} \to \BrPic(\ad(A_{2N+1}^{(n)}))$ that could give rise to a $\Z{M}$-graded extension of $\ad(A_{2N+1}^{(n)})$ $\otimes$-generated by an object of Frobenius-Perron dimension less than 2. Recall that the Brauer-Picard group of $\ad(A_{2N+1}^{(n)})$ is either $\Z{2}$ or $\Z{2}\times \Z{2}$, depending on whether $N$ is even or odd.

\begin{trivlist}\leftskip=2em
\item \textbf{Case $N$ even:}

Here the Brauer-Picard group is $\Z{2}$. As the only objects in the trivial bimodule $A_{2N+1}^\text{even}$ with dimension less than $2$ are invertible, we can ignore homomorphisms which map $1 \mapsto A_{2N+1}^\text{even}$, as such objects couldn't generate the entire category $\cC$. Thus we can assume $M$ even and $\Z{M} \to \BrPic(\ad(A_{2N+1}^{(n)}))$ is the map defined by $1 \mapsto {A_{2N+1}}^\text{odd}$.

\vspace{1em}

\item \textbf{Case $N$ odd:}

Here the Brauer-Picard group is $\Z{2}\times \Z{2}$. Exactly as in the $N$ even case we can rule out homomorphisms defined by $1 \mapsto A_{2N+1}^\text{even}$. 
The only time the bimodule $D_{N+2}^\text{odd}$ contains an object of dimension less than $2$ is when $N=3$, which we have excluded in this Lemma. Thus we can rule out homomorphisms defined by $1 \mapsto D_{N+2}^\text{odd}$.
The bimodule $D_{N+2}^\text{even}$ contains a single object of dimension less than $2$. However this object always has dimension $\sqrt{2}$, and could only generate the entire category $\cC$ when $N=1$, which we have also excluded in this Lemma. Therefore we can rule out homomorphisms defined by $1 \mapsto D_{N+2}^\text{even}$. Thus we can assume $M$ even and $\Z{M} \to \BrPic(\ad(A_{2N+1}^{(n)}))$ is the map defined by $1 \mapsto {A_{2N+1}}^\text{odd}$.
\end{trivlist}\vspace{1em}
Hence we can assume in either case that $M$ is even and $ \Z{M} \to \BrPic(\ad(A_{2N+1}^{(n)}))$ is the map defined by $1 \mapsto {A_{2N+1}}^\text{odd}$.

There are exactly two invertible elements in the centre of $\ad(A_{2N+1}^{(n)})$, and so we compute that
\[   H^2(\Z{M}, \Inv(\cZ(\ad(A_{2N+1}^{(n)})))) = \Z{2}.\] 
Thus, up to twisting the associator by an element of $H^3(\Z{M}, \mathbb{C}^\times)$, there are at most two $\Z{M}$-graded extensions of $\ad(A_{2N+1}^{(n)})$, $\otimes$-generated by an object of Frobenius-Perron dimension less than 2. By Lemma~\ref{lem:aOddExt} two such extensions are
\begin{align*}
& \subcat{A_{2N+1}^{(n)}\boxtimes \Vec( \Z{M} )  }  \text{ and,} \\
& \subcat{A_{2N+1}^{(n)}\boxtimes \LVec( \Z{2M} )  }_{ \langle f^{(2N)} \boxtimes M  \rangle    },
\end{align*}
when $N$ is even, and 
\begin{align*}
& \subcat{A_{2N+1}^{(n)}\boxtimes \Vec( \Z{M} )  }  \text{ and,} \\
& \subcat{A_{2N+1}^{(n)}\boxtimes \IVec( \Z{2M} )  }_{ \langle f^{(2N)} \boxtimes M  \rangle    },
\end{align*}
when $N$ is odd.

In either case, Lemma~\ref{lem:aOddDiff} gives that both extensions have different fusion rules, and thus are not equivalent up to twisting the associator. Hence these two categories, up to twisting the associator by an element of $H^3(\Z{M}, \mathbb{C}^\times)$, realise all $\Z{M}$-graded extensions of $\ad(A_{2N+1}^{(n)})$ generated by an object of Frobenius-Perron dimension less than 2.
\end{proof}

\subsection*{Cyclic extensions of the categories of adjoint $A_3$ type}
The adjoint $A_3$ case is different from the other adjoint $A_\text{odd}$ cases as there now exist 4 objects in $\Inv(\cZ(\ad(A_3^{(1)})))$ forming a $\Z{2} \times \Z{2}$ group. As is the general adjoint $A_\text{odd}$ case, we can restrict our attention to the homomorphism the $\Z{2M} \to \BrPic(\ad(A_3^{(1)}))$ defined by $1 \mapsto A_\text{3}^\text{odd}$. The bimodule $A_\text{3}^\text{odd}$ acts on $\Inv(\cZ(\ad(A_3^{(1)})))$ by exchanging the $\Z{2}$ factors. With this information we compute that 
\[    H^2(\Z{2M} , \Inv(\cZ(\ad(A_3^{(1)})))) = \begin{cases}
\{e\}   \text{ if } M \text{ is odd,}\\
\Z{2}  \text{ if } M \text{ is even.}
\end{cases} \]
Thus there are possible interesting $\Z{4M}$ graded extensions of $\ad(A_3^{(1)})$. In fact these interesting extensions exist, and are realised by the categories 
\[   \subcat{  A_3^{(1)} \boxtimes \IVec(\Z{8M})}_{ \langle f^{(2)} \boxtimes 4M \rangle }.\]
It is proven in Lemma~\ref{lem:aOddExt} that these categories are graded extensions of $\ad(A_3^{(1)})$, and Lemma~\ref{lem:aOddDiff} shows that these categories have different fusion rules to the categories $\subcat{  A_3^{(1)} \boxtimes \LVec(\Z{4M})}$. These facts allow us to prove the classification statement for this subsection.

\begin{lemma}\label{lem:a3}
If $\cC$ is a $\Z{M}$-graded extension of $\ad(A_3^{(1)})$, $\otimes$-generated by an object of Frobenius-Perron dimension less than 2, then either $M$ is even, and, up to twisting the associator of $\cC$ by an element of $H^3(\Z{M}, \mathbb{C}^\times)$, the category $\cC$ is monoidally equivalent to: 
\[
\subcat{ A_3^{(1)} \boxtimes \LVec(\Z{M}) },
\]
or $4$ divides $M$ and, up to twisting the associator of $\cC$ by an element of $H^3(\Z{M}, \mathbb{C}^\times)$, the category $\cC$ is monoidally equivalent to
\[   \subcat{  A_3^{(1)} \boxtimes \IVec(\Z{2M})}_{ \langle f^{(2)} \boxtimes M \rangle }.\]
\end{lemma}
\begin{proof}
Recall there are only two bimodule categories over $\ad(A_3^{(1)})$, which are $A_3^\text{even}$ and $A_3^\text{odd}$. As the bimodule $A_3^\text{even}$ only contains invertible objects, any extension generated by one of these objects would be pointed. Thus we can assume $M$ is even, and that the homomorphism $\Z{M} \to \BrPic(\ad(A_3^{(1)}))$ is defined by $1 \mapsto A_3^\text{odd}$. From the earlier discussion we have that $H^2(\Z{M} , \Inv(\cZ(\ad(A_3^{(1)})))$ is trivial when $M$ is not divisible by 4, and isomorphic to $\Z{2}$ when $M$ is divisible by 4.

 When $M$ is not divisible by 4 there is, up to twisting the associator, a unique $\Z{M}$-graded extension of $\ad(A_3^{(1)})$ $\otimes$-generated by an object of Frobenius-Perron dimension less than 2. This category is realised by
\[ \subcat{ A_3^{(1)} \boxtimes \LVec(\Z{M}) }.  \]

When $M$ is divisible by 4 there are, up to twisting the associator, two $\Z{M}$-graded extensions of $\ad(A_3^{(1)})$ $\otimes$-generated by an object of Frobenius-Perron dimension less than 2. These categories are realised by
\[ \subcat{ A_3^{(1)} \boxtimes \LVec(\Z{M}) }, \text{ and }  \subcat{  A_3^{(1)} \boxtimes \IVec(\Z{2M})}_{ \langle f^{(2)} \boxtimes M \rangle }.\]
\end{proof}

\subsection*{Cyclic extensions of categories of adjoint $E_6$ type}
The adjoint $E_6$ case is difficult for the same reason the $A_\text{odd}$ case was difficult, in that $\Inv(\cZ(\ad(E_6^{(n,\pm)})))$ is non-trivial, giving rise to interesting extensions. Fortunately, we can directly adapt the techniques used in the adjoint $A_\text{odd}$ case to deal with these difficulties. We begin this section by constructing an interesting $\Z{2M}$-graded extension of $\ad(E_6^{(n,\pm)})$.

The category $E_6^{(n,\pm)}$ contains the subcategory $\langle Z \rangle$, which lifts to a copy of $\sVec$ in the centre. Thus from Lemma~\ref{lem:liftyourself}, the subcategory $\langle Z\boxtimes 2M \rangle$ of $E_6^{(n,\pm)} \boxtimes \IVec(\Z{4M})$ lifts to a copy of $\Rep(\Z{2})$ in the centre, and hence the subcategory $\langle Z\boxtimes 2M \rangle$ of $\subcat{E_6^{(n,\pm)} \boxtimes \IVec(\Z{4M})}$ also lifts to a copy of $\Rep(\Z{2})$ in the centre. Thus we can de-equivariantize $\subcat{E_6^{(n,\pm)} \boxtimes \IVec(\Z{4M})}$ by the subcategory $\langle Z \boxtimes 2M  \rangle  $ to get the new fusion category
\[\subcat{E_6^{(n,\pm)}\boxtimes \IVec( \Z{4M} )  }_{ \langle Z \boxtimes 2M  \rangle    }.\]

The following two Lemmas show that this fusion category is an interesting extension of $\ad(E_6^{(n,\pm)})$. The proofs are direct translations of the proofs of the corresponding Lemmas in the $A_\text{odd}$ case.

\begin{lemma}\label{lem:E6Ext}
The category $\subcat{E_6^{(n,\pm)}\boxtimes \IVec( \Z{4M} )  }_{ \langle Z \boxtimes M  \rangle    }$ is a $\Z{2M}$ graded extension of $\ad(E_6^{(n,\pm)})$.
\end{lemma}
\begin{lemma}\label{lem:E6Diff}
The categories $\subcat{E_6^{(n,\pm)}\boxtimes \LVec( \Z{2M} )  }$ and $ \subcat{E_6^{(n,\pm)}\boxtimes \IVec( \Z{4M} )  }_{ \langle Z \boxtimes 2M  \rangle    }$ have different fusion rules.
\end{lemma}

With these two Lemmas we can prove the classification result for the adjoint $E_6$ case.
\begin{lemma}\label{lem:e6}
Fix $n\in \Z{12}^\times / \{\pm \}$. If $\cC$ is a $\Z{M}$-graded extension of $\ad(E_6^{(n,\pm)})$, $\otimes$-generated by an object of Frobenius-Perron dimension less than 2, then $M$ is even, and, up to twisting the associator of $\cC$ by an element of $H^3(\Z{M}, \mathbb{C}^\times)$, the category $\cC$ is monoidally equivalent to one of: 
\begin{align*}
& \subcat{E_6^{(n,\pm)}\boxtimes \LVec( \Z{M} )  }  \text{ or,} \\
& \subcat{E_6^{(n,\pm)}\boxtimes \IVec( \Z{2M} )  }_{ \langle Z \boxtimes M  \rangle    }.
\end{align*}
\end{lemma}
\begin{proof}
Recall the Brauer-Picard group of $\ad(E_6^{(n,\pm)})$ is $\Z{2}$. The only objects in the trivial bimodule $E_6^\text{even}$ with Frobenius-Perron dimension less than $2$ are invertible, and hence could not $\otimes$-generate the entire category $\cC$. Thus we can assume that $M$ is even and $\Z{M} \to \BrPic(\ad(E_6^{(n,\pm)}))$ is the map defined by $1 \mapsto E_6^\text{odd}$.

There are exactly two invertible elements in the centre of $\ad(E_6^{(n,\pm)})$, and so
\[  H^2(\Z{M}, \Inv(\cZ(\ad(E_6^{(n,\pm)})))) = \Z{2}.\] 
Thus, up to twisting the associator by an element of $H^3(\Z{M}, \mathbb{C}^\times)$, there are at most two $\Z{M}$-graded extensions of $\ad(E_6^{(n,\pm)})$ generated by an object of Frobenius-Perron dimension less than 2. By Lemma~\ref{lem:E6Ext} two such extensions are
\begin{align*}
&  \subcat{E_6^{(n,\pm)}\boxtimes \LVec( \Z{M} )  }  \text{ and,} \\
& \subcat{E_6^{(n,\pm)}\boxtimes \IVec( \Z{2M} )  }_{ \langle Z \boxtimes M  \rangle        }.
\end{align*}
By Lemma~\ref{lem:E6Diff}, both of these extensions have different fusion rules, and thus are not monoidally equivalent, even up to twisting the associator. Hence these two categories, up to twisting the associator by an element of $H^3(\Z{M}, \mathbb{C}^\times)$, realise all $\Z{M}$-graded extensions of $\ad(E_6^{(n,\pm)})$ $\otimes$-generated by an object of Frobenius-Perron dimension less than 2.
\end{proof}

\subsection*{Cyclic extensions of $\ad(D_4)$}
As with the adjoint $A_3$ case, the adjoint $D_4$ case is interesting as there exist 9 objects in 
\[ \Inv(\cZ(\ad(D_4^{(n)}))) =  \ad(D_4^{(n)})\boxtimes \ad(D_4^{(n)})^\text{bop},\]
 forming a $\Z{3}\times \Z{3}$ group. Hence there is the possibility for the group
 \[  H^2(\Z{M} , \Inv(\cZ(\ad(D_4^{(n)}))))\]
 to be non-trivial, which implies the possible existence of interesting $\Z{M}$-graded extensions of $\ad(D_4^{(n)})$. We start this subsection by constructing a family of interesting $\Z{6M}$-graded extensions of $\ad(D_4^{(n)})$.
 
 From \cite{MR3808052} we know that the subcategory $\langle P \rangle$ of $D_4^{(n,\pm)}$ lifts to a copy of $\Rep(\Z{3})$ in the centre. Consider the subcategory $ \langle 2M \rangle$ of $\LVec(\Z{6M})$. This subcategory lifts to a copy of $\Rep(\Z{3})$ in the centre. Thus the subcategory $\langle P \boxtimes 2M\rangle$ of $D_4^{(n,\pm)}\boxtimes \LVec(\Z{6M})$ also lifts to a copy of $\Rep(\Z{3})$ in the centre, and hence the subcategory $\langle P \boxtimes 2M\rangle$ of $\subcat{D_4^{(n,\pm)}\boxtimes \LVec(\Z{6M})}$ also lifts to a copy of $\Rep(\Z{3})$ in the centre. Thus we can de-equivariantize $\subcat{D_4^{(n,\pm)}\boxtimes \LVec(\Z{6M})}$ by the subcategory $\langle P \boxtimes 2M\rangle$ to get the new fusion category
 \[ \subcat{D_4^{(n,\pm)}\boxtimes \LVec(\Z{6M})}_{\langle P \boxtimes 2M\rangle}.  \]
 
 In the following Lemmas we show that when $M$ is divisible by 3, these categories are interesting extensions of $\ad(D_4^{(n)})$.
 
 \begin{lemma}\label{lem:d4plusext}
 Fix $\kappa$ a choice of sign, then the category $\subcat{ D_4^{(n,\kappa)}\boxtimes \LVec(\Z{6M}})_{\langle P \boxtimes 2M \rangle}$ is a $\Z{M}$ graded extension of $D_4^{(\pm n,\kappa)}$.
\end{lemma}
\begin{proof}
Direct adaptations of previous arguments (i.e. Lemma~\ref{lem:aOddExt}) show that the category $\subcat{ D_4^{(n,+)}\boxtimes \Vec(\Z{6M})}_{\langle P \boxtimes 2M \rangle}$ is a $\Z{M}$-graded extension of $D_4^{(\pm n,\mu)}$ for $\mu$ an undetermined sign. The subtlety of this argument comes in determining $\mu$.

The categories $ D_4^{(n,\pm)}$ contain a morphism $S \in \operatorname{Hom}( {f^{(1)}}^{\otimes 4} \to \mathbf{1})$, distinguished (up to scaler) by the property that the one click rotation of $S$ is equal to $\pm i S$. Thus the category $\subcat{ D_4^{(n,\kappa)}\boxtimes \LVec(\Z{6M})}$ contains the morphism $S \boxtimes \id_0 : ( f^{(1)} \boxtimes 9 )^{\otimes 4} \to \mathbf{1}\boxtimes 0$ whose one click rotation scales by $\kappa i$. The free module functor
\[  \subcat{ D_4^{(n,\kappa)}\boxtimes \LVec(\Z{6M}}) \to \subcat{ D_4^{(n,\kappa)}\boxtimes \LVec(\Z{6M}})_{\langle P \boxtimes 2M \rangle}\]
sends the morphism $S \boxtimes \id_0$, to a morphism $f$ that lands in the $D_4^{(\pm n,\mu)}$ subcategory of $\subcat{ D_4^{(n,\kappa)}\boxtimes \LVec(\Z{6M})}_{\langle P \boxtimes 2M \rangle}$. As the morphism $f$ is the image of $S \boxtimes \id_0$, it also has the property that it is a rotational eigenvector, with rotational eigenvalue $\kappa i$. Hence we have that $\mu = \kappa$.
\end{proof}

\begin{lemma}\label{lem:d4plusdif}
Suppose $M$ is divisible by 3, then the categories $ \subcat{ D_4^{(n,\pm)}\boxtimes \LVec(\Z{2M}})$ and $\subcat{ D_4^{(n,\pm)}\boxtimes \LVec(\Z{6M}})_{\langle P \boxtimes 2M \rangle}$ have different fusion rules.
\end{lemma}
\begin{proof}
It is direct to see that the invertible elements of $ \subcat{ D_4^{(n,\pm)}\boxtimes \LVec(\Z{2M}})$ form a group isomorphic to $\Z{3} \times \Z{M}$.

On the other hand the category $\subcat{ D_4^{(n,\pm)}\boxtimes \LVec(\Z{6M}})_{\langle P \boxtimes 2M \rangle}$ has $3M$ invertible elements. One of these is the element
\[ (\mathbf{1} \boxtimes 2) \oplus (P \boxtimes 2M+2)\oplus (Q \boxtimes 4M+2).  \]
A direct computation shows that this object has order 3M. Hence the invertible elements of $\subcat{ D_4^{(n,\pm)}\boxtimes \LVec(\Z{6M}})_{\langle P \boxtimes 2M \rangle}$ form a group isomorphic to $\Z{3M}$.

As $M$ is divisible by 3, the groups $\Z{3}\times\Z{M}$ and $\Z{3M}$ are non-isomorphic. Thus the categories $ \subcat{ D_4^{(n,\pm)}\boxtimes \LVec(\Z{2M}})$ and $\subcat{ D_4^{(n,\pm)}\boxtimes \LVec(\Z{6M}})_{\langle P \boxtimes 2M \rangle}$ have different fusion rules.
\end{proof}

We are now placed to prove the classification result of this subsection.

\begin{lemma}\label{lem:d4}
Fix $n\in \Z{6}^\times / \{\pm \}$. If $\cC$ is a $\Z{M}$-graded extension of $\ad(D_4^{(n)})$, $\otimes$-generated by an object of Frobenius-Perron dimension less than 2, then either $M$ is even, and, up to twisting the associator of $\cC$ by an element of $H^3(\Z{M}, \mathbb{C}^\times)$, the category $\cC$ is monoidally equivalent to one of: 
\[
\subcat{ D_4^{(n,\pm)} \boxtimes \LVec(\Z{M}) },
\]
or $6$ divides $M$ and, up to twisting the associator of $\cC$ by an element of $H^3(\Z{M}, \mathbb{C}^\times)$, the category $\cC$ is monoidally equivalent to
\[   \subcat{  D_4^{(n,\pm)} \boxtimes \LVec(\Z{3M})}_{ \langle P \boxtimes M \rangle }.\]
\end{lemma}
\begin{proof}
As in the general adjoint $D_{2N}$ case, we can restrict to $M$ even, and the homomorphisms $\Z{M} \to \BrPic(\ad(D_4^{(n)}))$ defined by
\[ 1 \mapsto D_4^\text{odd}\text{ and } 1 \mapsto _{P\leftrightarrow Q}D_4^\text{odd}.  \]

We split the proof in to two cases.

\begin{trivlist}\leftskip=2em
\item{ \textbf{Case $1 \mapsto D_4^\text{odd}$:}

With this choice of homomorphism we compute that 
\[  H^2(\Z{M} , \Inv(\cZ(\ad(D_4^{(n)}))))  = \begin{cases}
\{e\} \text{ if } M \text{ is not divisible by 6},\\
\Z{3} \text{ if } M \text{ is divisible by 6}.
\end{cases}\]
Representatives of the non-trivial cocycles are given by:
\[    T_P(n,m) :=   \begin{cases}
\mathbf{1} \boxtimes \mathbf{1} \text{ if } n+m < M\\
\mathbf{1}\boxtimes P  \text{ if } n+m \geq M,
\end{cases} 
\]
and
\[    T_Q(n,m) :=   \begin{cases}
\mathbf{1} \boxtimes \mathbf{1} \text{ if } n+m < M\\
\mathbf{1}\boxtimes Q  \text{ if } n+m \geq M.
\end{cases} 
\]
The group of monoidal auto-equivalences of $\ad(D_4^{(n)})$ acts on $H^2(\Z{M} , \Inv(\cZ(\ad(D_4^{(n)}))))$ via the map 
\begin{align*}
      \TenAut(\ad(D_4^{(n)}) ) &\to \BrPic(    \ad(D_4^{(n)})   ) \\
      \mathcal{F}\quad &\mapsto\quad   _\mathcal{F}D_4^\text{even}
\end{align*}
and the standard action of $\BrPic( \ad(D_4^{(n)}))$ on $\Inv( \cZ ( \ad(D_4^{(n)})))$ described in Table~\ref{tab:Inv}. We compute that the monoidal auto-equivalence $P\leftrightarrow Q$ of $\ad(D_4^{(n)})$ acts on the cocycle $T_P$ to give the cocyle $T_Q$. Thus, up to action of $\TenAut(\ad(D_4^{(n)}))$, there are exactly two elements of $H^2(\Z{M} , \Inv(\cZ(\ad(D_4^{(n)}))))$ when $M$ is divisible by 6, and one element otherwise. Thus, as $\BrPic(\ad(D_4^{(n)}))$ is abelian and $| \TenAut(\ad(D_4^{(n)})) | = 2$ is coprime to $| H^2(\Z{M} , \Inv(\cZ(\ad(D_4^{(n)})))) |=3$, we can apply \cite[Theorem 3.1]{1711.00645} to see that up to monoidal equivalence and twisting the associator, there are at most two $\Z{M}$-graded extensions of $\ad(D_4^{(n)})$ corresponding to the homomorphism $1 \mapsto D_4^\text{odd}$ when $M$ is divisible by 6, and only one extension otherwise.

When $M$ is not divisible by 6, the unique $\Z{M}$-graded extension is realised by the category
\[   \subcat{ D_4^{(n,+)} \boxtimes \LVec(\Z{M}) }.  \]

When $M$ is divisible by 6, the two $\Z{M}$-graded extensions are realised by the categories
\[   \subcat{ D_4^{(n,+)} \boxtimes \LVec(\Z{M}) } \text{ and }  \subcat{ D_4^{(n,+)}\boxtimes \LVec(\Z{6M}})_{\langle P \boxtimes 2M \rangle}.  \]
The latter category is a $\Z{M}$-graded extension of $ \ad(D_4^{(n)})$ that corresponds to the homomorphism $1 \mapsto D_4^\text{odd}$ by Lemma~\ref{lem:d4plusext}. These two categories are monoidally inequivalent, even up to twisting the associator, by Lemma~\ref{lem:d4plusdif}.}

\vspace{2em}\item{
\textbf{Case $1 \mapsto _{P\leftrightarrow Q}D_4^\text{odd}$:}

With this choice of homomorphism we again compute that 
\[  H^2(\Z{M} , \Inv(\cZ(\ad(D_4^{(n)}))))  = \begin{cases}
\{e\} \text{ if } M \text{ is not divisible by 6},\\
\Z{3} \text{ if } M \text{ is divisible by 6}.
\end{cases}\]
We now have different representatives of the non-trivial cocycles, given by
\[    V_P(n,m) :=   \begin{cases}
\mathbf{1} \boxtimes \mathbf{1} \text{ if } n+m < M\\
P \boxtimes \mathbf{1}  \text{ if } n+m \geq M,
\end{cases} 
\]
and
\[    V_Q(n,m) :=   \begin{cases}
\mathbf{1} \boxtimes \mathbf{1} \text{ if } n+m < M\\
Q \boxtimes \mathbf{1}  \text{ if } n+m \geq M.
\end{cases} 
\]
 The monoidal auto-equivalence $P\leftrightarrow Q$ of $\ad(D_4^{(n)})$ acts on $V_P$ to give $V_Q$. Thus we again have that up to monoidal equivalence and twisting the associator, there are at most two $\Z{M}$-graded extensions of $\ad(D_4^{(n)})$ corresponding to the homomorphism $1 \mapsto _{P\leftrightarrow Q}D_4^\text{odd}$ when $M$ is divisible by 6, and only one extension otherwise.

 When $M$ is not divisible by 6, the unique $\Z{M}$-graded extension is realised by the category
\[   \subcat{ D_4^{(n,-)} \boxtimes \LVec(\Z{M}) }.  \]

 When $M$ is divisible by 6, the two $\Z{M}$-graded extensions are realised by the categories
\[   \subcat{ D_4^{(n,-)} \boxtimes \LVec(\Z{M}) } \text{ and }  \subcat{ D_4^{(n,-)}\boxtimes \LVec(\Z{6M}})_{\langle P \boxtimes 2M \rangle}.  \]}
\end{trivlist}
\end{proof}

\subsection*{Cyclic extensions of categories of adjoint $D_{10}$ type}
The adjoint $D_{10}$ case proves to be one of the most interesting cases, as the Brauer-Picard group is $S_3 \times S_3$, and thus there are many interesting cyclic subgroups. In particular there are $\Z{6}$ subgroups, which suggests the possible existence of an exotic extension of $\ad(D_{10}^{(n)})$. In this subsection we show that this $\Z{6}$-graded extension does in fact exist, and is realised by $\mathcal{E}_{16,6}$, the exceptional quantum subgroup of $\mathfrak{sl}_2$ at level $16$ crossed with $\mathfrak{sl}_3$ at level 6. To begin this subsection we give a construction of the categories $\mathcal{E}_{16,6}$.

Consider the category $\cC( \mathfrak{sl}_2, 16) \boxtimes \cC( \mathfrak{sl}_3, 6)$ (adopting the notation of \cite[Section 2.3]{MR3808050}). Via the conformal embedding $\mathfrak{sl}(2)_{16} \oplus \mathfrak{sl}(3)_{6} \subset (E_8)_1$, along with \cite[Theorem 5.2]{MR1936496}, there exists a commutative algebra object $A^{(1)} \in \cC( \mathfrak{sl}_2, 16) \boxtimes \cC( \mathfrak{sl}_3, 6)$. The category $\cC( \mathfrak{sl}_2, 16) \boxtimes \cC( \mathfrak{sl}_3, 6)$ is defined over the field $\mathbb{Q}[\zeta_{36}]$. Applying the Galois automorphisms $\zeta_{36}\mapsto \zeta_{36}^n$, for $n \in \{1,5,7\}$ to $\cC( \mathfrak{sl}_2, 16)\boxtimes \cC( \mathfrak{sl}_3, 6)$ gives 3 categories with the same fusion rules as $\cC( \mathfrak{sl}_2, 16) \boxtimes \cC( \mathfrak{sl}_3, 6)$. Furthermore, the Galois automorphism caries the commutative algebra $A^{(1)}$ to a commutative algebra $A^{(n)}$ in each of these categories.
As these Galois conjugates of $\cC( \mathfrak{sl}_2, 16) \boxtimes \cC( \mathfrak{sl}_3, 6)$ are all modular, there exist two central structures on the algebra objects $A^{(n)}$, which we simply call $\pm$. Thus the category of $A$-modules has the structure of a fusion category by \cite{MR2863377}.
 \begin{dfn}\label{def:e166}
 We write $\mathcal{E}^{(n,\pm)}_{16,6}$ for the fusion category of $(A^{(n)},\pm)$-modules in the appropriate Galois conjugate of $\cC( \mathfrak{sl}_2, 16) \boxtimes \cC( \mathfrak{sl}_3, 6)$.
 \end{dfn}

We now show that the categories $\mathcal{E}^{(n,\pm)}_{16,6}$ realise the interesting possible $\Z{6}$-graded extensions of $\ad(D_{10}^{(n)})$

\begin{lemma}\label{lem:D10ext}
The categories $\mathcal{E}^{(n,\pm)}_{16,6}$ are $\Z{3}$-graded extensions of $D_{10}^{(n,\pm)}$, for $n = \{1,5,7\}$.
\end{lemma}
\begin{proof}
We first compute the categorical dimensions of the simple objects of $\mathcal{E}^{(n,\pm)}_{16,6}$. The remark in the proof of \cite[Theorem 6]{MR1976459}, states that as a module over $\cC( \mathfrak{sl}_2, 16)$, the category $\mathcal{E}^{(1,+)}_{16,6}$ is a sum of a $D_{10}$ module, and two $E_7$ modules. This fact allows us to compute the Frobenius-Perron dimensions of the simple objects of $\mathcal{E}^{(n,\pm)}_{16,6}$. Coupled with the fact that the category $\mathcal{E}^{(n,\pm)}_{16,6}$ takes a functor from the appropriate Galois conjugate of $\cC( \mathfrak{sl}_2, 16) \boxtimes \cC( \mathfrak{sl}_3, 6)$, we get that the categorical dimensions of the 24 simple objects of $\mathcal{E}^{(n,\pm)}_{16,6}$ are, for $q = e^{n\frac{2i\pi}{18}}$:
\begin{align*}
\{&1 , [2]_q,[3]_q,[4]_q,[5]_q,[6]_q,[7]_q,[8]_q,[9]_q,[3]_q,[3]_q, \\
 &[2]_q, [3]_q+1,[4]_q+[2]_q,[5]_q+[3]_q,[2]_q + q^3 + q^{-3},[6]_q, 1 + q^4 + q^{-4}, \\
&[2]_q, [3]_q+1,[4]_q+[2]_q,[5]_q+[3]_q,[2]_q + q^3 + q^{-3},[6]_q, 1 + q^4 + q^{-4} \}.
\end{align*}
In particular we see that there are three objects of dimension $[3]_q$, thus at least one of these objects is self-dual. Let $X$ be this self-dual object, and $\cC_X$ the fusion subcategory of $\mathcal{E}^{(n,\pm)}_{16,6}$  generated by $X$. By the classification of undirected graphs of norm less than 2, the fusion graph for $X\in \cC_X$ must be either the $A_{17}$ or $D_{10}$ graph. If this fusion graph was $A_{17}$, then there must be an object of Frobenius-Perron dimension $\approx 5.75877$ in $\cC_X$, and thus in $\mathcal{E}^{(n,\pm)}_{16,6}$. However consulting the above list of categorical dimensions in $\mathcal{E}^{(n,\pm)}_{16,6}$ shows that there is no such object, thus the fusion graph for $X \in \cC_X$ must be the $D_{10}$ graph.

The category $\mathcal{E}^{(n,\pm)}_{16,6}$ is pivotal \cite[Theorem 1.17]{MR1936496}, and thus the subcategory $\cC_X$ inherits a pivotal structure. As the object $X$ tensor generates $\cC_X$, and the $D_{10}$ graph is bipartite, we must have that $\cC_X$ is $\Z{2}$-graded, thus there exist two pivotal structures on $\cC_X$. With respect to one of these pivotal structures, the object $X$ is symmetrically self-dual. Thus by Theorem~\ref{thm:ADET}, $\cC_X$ is equivalent to the category $D_{10}^{(v,\delta)}$ for some $v \in \Z{18}^\times$ and choice of sign $\delta$. As the categorical dimension of $X$ is $[2]_q$ we can deduce that $v = \pm n$. 

Finally we have to show that the $\pm$ sign in $\mathcal{E}^{(n,\pm)}_{16,6}$ agrees with $\delta$. For this we use Frobenius-Schur indicators, defined in \cite{MR2313527}. Again we pick out the unique symmetrically self-dual object $X$. This corresponds to the object $f^{(1)}$ in the subcategory $D_{10}^{(v,\delta)}$. Using the planar algebra presentation of the category $D_{10}^{(v,\delta)}$ \cite{MR2559686}, we can easily compute that the $4N+4$-th Frobenius-Schur indicator of the object $f^{(1)}$ as
\[   \nu_{4N+4}(f^{(1)}) = \delta i.\]
Hence we must also have in $\mathcal{E}^{(n,\pm)}_{16,6}$ that 
\[   \nu_{4N+4}(X) = \delta i.\]
Using \cite[Theorem 4.1]{MR2313527} we can also express $\nu_{4N+4}(X)$ in terms of the modular data of $\cZ(\mathcal{E}^{(n,\pm)}_{16,6})$. As the categories $\mathcal{E}^{(1,\pm)}_{16,6}$ are quantum subgroups of $\cC( \mathfrak{sl}_2, 16) \boxtimes \cC( \mathfrak{sl}_3, 6)$, we can use \cite[Corollary 3.30]{MR3039775} to show that
\[ \cZ(\mathcal{E}^{(1,+)}_{16,6}) = \cC( \mathfrak{sl}_2, 16) \boxtimes \cC( \mathfrak{sl}_3, 6)^\text{bop} \text{ and } \cZ(\mathcal{E}^{(1,-)}_{16,6}) = \cC( \mathfrak{sl}_2, 16)^\text{bop} \boxtimes \cC( \mathfrak{sl}_3, 6).\]
The Drinfeld centres of $ \cZ(\mathcal{E}^{(n,\pm)}_{16,6})$ for $n \in \{5,7\}$ are obtained by applying the Galois automorphism $\zeta_{36}\mapsto \zeta_{36}^n$ to the braided categories $ \cZ(\mathcal{E}^{(1,\pm)}_{16,6})$.
A calculation now shows that  $ \nu_{4N+4}(X) = \pm i$. Hence $\delta$ agrees with the sign of $\mathcal{E}^{(n,\pm)}_{16,6}$.
\end{proof}
While not a-priori, the categories $\mathcal{E}^{(n,\pm)}_{16,6}$ contain a $\otimes$-generating object of Frobenius-Perron dimension $2\cos\left(\frac{\pi}{18}\right)$. We prove this in Appendix~\ref{app:rules}, where we deduce the fusion rules for the categories $\mathcal{E}^{(n,\pm)}_{16,6}$.

Now that we have shown the categories $\mathcal{E}^{(n,\pm)}_{16,6}$ realise interesting $\Z{6}$-graded extensions of $\ad(D_{10}^{(n)})$, we can prove the main classification result of this section.
\begin{lemma}\label{lem:d10}
Fix $n\in \Z{18}^\times / \{\pm \}$. If $\cC$ is a $\Z{M}$-graded extension of $\ad(D_{10}^{(n)})$, $\otimes$-generated by an object of Frobenius-Perron dimension less than 2, then either $M$ is even, and, up to twisting the associator of $\cC$ by an element of $H^3(\Z{M}, \mathbb{C}^\times)$, the category $\cC$ is monoidally equivalent to: 
\[
\subcat{ D_{10}^{(n,\pm)} \boxtimes \LVec(\Z{M}) },
\]
or $6$ divides $M$ and, up to twisting the associator of $\cC$ by an element of $H^3(\Z{M}, \mathbb{C}^\times)$, the category $\cC$ is monoidally equivalent to
\[   \subcat{\mathcal{E}^{(n,\pm)}_{16,6} \boxtimes \LVec(\Z{M})}.\]
\end{lemma}
\begin{proof}
We begin by classifying homomorphisms from the cyclic group $\Z{M}$ to $\BrPic(\ad(D^{(n)}_{10}))$ that may give rise to extensions generated by an object of dimension less than 2. Consulting the table of dimensions from Subsection~\ref{sub:dims} shows that the only bimodules over $\ad(D^{(n,+)}_{10})$ that contain an object of Frobenius-Perron dimension less than 2, are $D_{10}^{\text{odd}}$, ${E_7}^{\text{even}}$, and ${\overline{E_7}}^{\text{even}}$, along with the twistings of each by the 5 non-trivial auto-equivalences of $\ad(D^{(n,+)}_{10})$. This leaves us with a total of 18 homomorphisms to consider. Fortunately \cite[Theorem 3.1]{1711.00645} shows that to get a representative from each monoidal equivalence class, we only need to consider homomorphisms $\Z{M} \to \BrPic(\ad(D^{(n,+)}_{10}))$, considered up to post-composition by the inner automorphisms coming from one of the six bimodules $D_{10}^\text{even}$, $_{P\leftrightarrow Q} D_{10}^\text{even}$, $_{P\leftrightarrow f^{(2)}} D_{10}^\text{even}$, $_{f^{(2)} \leftrightarrow Q} D_{10}^\text{even}$, $_{f^{(2)}\mapsto P\mapsto Q} D_{10}^\text{even}$, and $_{f^{(2)}\mapsto Q \mapsto P} D_{10}^\text{even}$. As the group structure of $\BrPic(\ad(D^{(n,+)}_{10}))$ has been described in \cite{MR3808052}, we can directly compute that we only have to consider the 4 homomorphisms:
\begin{align*}
1\mapsto & D_{10}^{\text{odd}},\\
1\mapsto & _{P\leftrightarrow Q}D_{10}^{\text{odd}},\\
1\mapsto & _{f^{(2)} \rightarrow Q\rightarrow P}{E_7}^{\text{even}},\\
1\mapsto & _{P \leftrightarrow Q}{E_7}^{\text{even}}.
\end{align*}

We finish our proof in two cases.

\begin{trivlist}\leftskip=2em
\item \textbf{Case $M$ is even and $1\mapsto  D_{10}^{\text{odd}}$ or $_{P\leftrightarrow Q}D_{10}^{\text{odd}}$:}

As $H^2(\Z{M} , \Inv(\cZ(\ad(D^{(n)}_{10})))) = \{e\}$, there is a unique extension, up to twisting the associator, corresponding to each of the two homomorphisms. These extensions are realised by the categories
\[
\subcat{ D_{10}^{(n,+)} \boxtimes \LVec(\Z{M}) } \text{ and } \subcat{ D_{10}^{(n,-)} \boxtimes \LVec(\Z{M}) }.
\]
These two categories are non-equivalent, even up to twisting the associator, by Lemma~\ref{lem:d2ndiff}.

\vspace{1em}

\item \textbf{Case 6 divides $M$ and $1\mapsto _{f^{(2)} \rightarrow Q\rightarrow P}{E_7}^{\text{even}}$ or $_{P \leftrightarrow Q}{E_7}^{\text{even}}$:}

Again, as $H^2(\Z{M} , \Inv(\cZ(\ad(D^{(n)}_{10})))) = \{e\}$, there is a unique extension, up to twisting the associator, corresponding to each of the two homomorphisms. These two extensions are realised by the categories
\[   \subcat{\mathcal{E}^{(n,+)}_{16,6} \boxtimes \LVec(\Z{M})} \text{ and }  \subcat{\mathcal{E}^{(n,-)}_{16,6} \boxtimes \LVec(\Z{M})}.\]
By Lemma~\ref{lem:D10ext} these categories are $\Z{3}$-graded extensions of $D_{10}^{(n,+)}$ and $D_{10}^{(n,-)}$ respectively. Thus both categories are $\Z{6}$-graded extensions of $\ad(D_{10}^{(n)})$. 

Aiming towards a contradiction, suppose that the categories 
\[ \subcat{\mathcal{E}^{(n,+)}_{16,6} \boxtimes \LVec(\Z{M})} \text{ and } \subcat{\mathcal{E}^{(n,-)}_{16,6} \boxtimes \LVec(\Z{M})}\]
are equivalent up to twisting the associator. Then this would imply that the respective subcategories $D_{10}^{(n,+)}$ and $D_{10}^{(n,-)}$ are equivalent up to twisting the associator. However this is a contradiction to Lemma~\ref{lem:d2ndiff}.
\end{trivlist}
\end{proof}

\subsection*{Cyclic extensions of categories of adjoint $A_7$ type}
The adjoint $A_7$ case is one of the most interesting cases, as the Brauer-Picard group of $\ad(A_7^{(n)})$ is isomorphic to $D_{2\cdot4}$, and thus contains a cyclic subgroup of order 4. The possible extention of $\ad(A_7^{(n)})$ can not be constructed through Deligne products and de-equivariantizations, as in the previous cases. Instead we construct this interesting extension as a quantum subgroup of $\mathfrak{sl}_4$. We begin this subsection by constructing this quantum subgroup, and showing it is a $\Z{4}$-graded extension of $\ad(A_7^{(n)})$.

Consider the category $\cC( \mathfrak{sl}_4, 4)$ (again adopting the notation of \cite[Section 2.3]{MR3808050}). Via the conformal embedding $\mathfrak{sl}(4)_4 \subset \mathfrak{spin}(15)_1$, along with  \cite[Theorem 5.2]{MR1936496} there exists a commutative algebra object $A^{(1)} \in \cC( \mathfrak{sl}_4, 4)$. The category $\cC( \mathfrak{sl}_4, 4)$ is defined over the field $\mathbb{Q}[\zeta_{16}]$, and applying the Galois automorphism $\zeta_{16}\mapsto \zeta_{16}^3$ gives another category with the same fusion rules. The algebra $A^{(1)}$ is carried by this Galois automorphism, and gives a commutative algebra $A^{(3)}$ in the Galois conjugate of $\cC( \mathfrak{sl}_4, 4)$.

 \begin{dfn}\label{def:e4}
 We write $\mathcal{E}^{(n)}_{4}$ for the fusion category of $ A^{(n)}$-modules in the appropriate Galois conjugate of $\cC( \mathfrak{sl}_4, 4)$.
 \end{dfn}

We now show that the fusion categories $\mathcal{E}^{(n)}_4$ are graded extensions of the fusion categories of adjoint $A_7$ type.

\begin{lemma}\label{lem:a7ext}
The categories $\mathcal{E}^{(n)}_{4}$ are $\Z{4}$-graded extensions of $\ad(A_7^{(n)})$.
\end{lemma}
\begin{proof}
We begin by computing the categorical dimensions of the simple objects of $\mathcal{E}^{(n)}_{4}$. Using the forgetful functor $\mathcal{E}^{(n)}_{4}\to \cC( \mathfrak{sl}_4, 4)$ \cite{Su4} we compute these dimensions as:
\[  \{1, [3]_q,[3]_q,1, [2]_q, [2]_q, (q^2 + q^{-2})[2]_q, q^2 + q^{-2} , [3]_q + 1,[2]_q, [2]_q, (q^2 + q^{-2})[2]_q\},    \]
 for $q = e^{n \frac{2 \pi i}{16}}$. In particular, the dimensions the of the $0$-graded piece of this category are
\[ \{1, 1 \pm \sqrt{2}, 1 \pm \sqrt{2}, 1\}. \]
There are precisely two fusion rings compatible with these dimensions. One is the $\ad(A_7)$ fusion ring, while the other is $\operatorname{HL}(c = 1)$, the fusion ring from \cite{MR3229513} with choice of $c=1$.

Aiming for a contradiction, suppose that the 0-graded piece of $\mathcal{E}^{(n)}_{4}$ has fusion ring $\operatorname{HL}(c = 1)$. As the 1-graded piece of the category $\mathcal{E}^{(n)}_{4}$ has rank 3, there exists a rank 3 module category over the $0$-graded piece of $\mathcal{E}^{(n)}_{4}$. However the techniques of \cite{MR2909758} (which rely only on the underlying fusion ring) show that any category with $\operatorname{HL}(c = 1)$ fusion ring does not have a rank 3 module. Hence we have our contradiction, and thus the fusion ring of the 0-graded piece of $\mathcal{E}^{(n)}_{4}$ is the $\ad(A_7)$ fusion ring.

We now use Lemma~\ref{lem:adclass} to see that the 0-graded piece of $\mathcal{E}^{(n)}_{4}$ is equivalent to $\ad(A_7^{(m)})$ for some $m\in \Z{8}^\times / \{\pm\}$. Considering categorical dimensions shows that $m = n$.
\end{proof}
We compute the fusion rules for the categories $\mathcal{E}_4^{(n)}$ in Appendix~\ref{app:rules}. Following Appendix~\ref{app:rules}, we label the 12 simple objects of $\mathcal{E}_4^{(n)}$ by the numbers $\mathbf{1}$ through $\mathbf{12}$. We see from this Appendix that the object $\mathbf{5}$ $\otimes$-generates $\mathcal{E}_4^{(n)}$, and has Frobenius-Perron dimension $\sqrt{2+\sqrt{2}}$.

We can identify the $\ad(A_7^{(n)})$ subcategory of $\mathcal{E}_4^{(n)}$ in two different ways, either by the map
\[  \mathbf{1} \mapsto f^{(0)} ,\quad \mathbf{2} \mapsto f^{(2)}, \quad \mathbf{3} \mapsto f^{(4)}, \quad \text{ and } \mathbf{4} \mapsto f^{(6)},\]
or the map 
\[  \mathbf{1} \mapsto f^{(0)} , \quad \mathbf{2} \mapsto f^{(4)},\quad \mathbf{3} \mapsto f^{(2)},\quad \text{ and } \mathbf{4} \mapsto f^{(6)}.\]
These two identifications give the category $\mathcal{E}_4^{(n)}$ the structure of a $\Z{4}$-graded extension of $\ad(A_7^{(n)})$ in two different ways. We write $( \mathcal{E}_4^{(n)}, +)$ and $(\mathcal{E}_4^{(n)},-)$ for these two extensions, respectively. It turns out that these two extensions are different despite being constructed from the same monoidal category.
\begin{lemma}\label{lem:dobcov}
The extensions $( \mathcal{E}_4^{(n)}, +)$ and $(\mathcal{E}_4^{(n)},-)$ are inequivalent as $\Z{4}$-graded extensions of $\ad(A_7^{(n)})$.
\end{lemma}
\begin{proof}
Recall that an equivalence of graded extensions is a monoidal equivalence that is the identity on the common trivial piece, and also preserves the grading group. An equivalence of extensions between $( \mathcal{E}_4^{(n)}, +)$ and $(\mathcal{E}_4^{(n)},-)$ would thus have to map $\mathbf{2} \leftrightarrow \mathbf{3}$. However Lemma~\ref{lem:E4auto} shows that any possible monoidal auto-equivalence of $\mathcal{E}_4^{(n)}$ that sends $\mathbf{2} \leftrightarrow \mathbf{3}$, always induces the group automorphism $1 \leftrightarrow 3$ on the grading group $\Z{4}$. Thus there can be no equivalence of graded extensions between $( \mathcal{E}_4^{(n)}, +)$ and $(\mathcal{E}_4^{(n)},-)$.
\end{proof}
\begin{cor}\label{lem:dobgencov}
When $M$ is odd, the category $\subcat{ \mathcal{E}^{(n)}_{4}\boxtimes \LVec(\Z{4M}})$ realises two distinct extensions of $\ad(A_7^{(n)})$.
\end{cor}
\begin{proof}
When $M$ is odd there exists a fully faithful functor of extensions
\[  \mathcal{E}_4^{(n)}\to  \subcat{ \mathcal{E}^{(n)}_{4}\boxtimes \LVec(\Z{4M}})\]
given by
\[  X_i \mapsto X_i\boxtimes -mi.    \]
The result then follows from Lemma~\ref{lem:dobcov}.
\end{proof}

As $\Inv(\cZ(\ad(A_7^{(n)}))) = \Z{2}$, we have that for each fixed homomorphism $\Z{M} \to \BrPic(\ad(A_7^{(n)}))$, there are two corresponding extensions, up to twisting the associator. As mentioned earlier, the group $\BrPic(\ad(A_7^{(n)}))$ contains an interesting $\Z{4}$ subgroup, and hence there are interesting homomorphisms $\Z{4M} \to  \BrPic(\ad(A_7^{(n)}))$. We have just shown that the categories $\subcat{ \mathcal{E}^{(n)}_{4}\boxtimes \LVec(\Z{4M}})$ realise these interesting homomorphisms. When $M$ is odd we show in Corollary~\ref{lem:dobgencov} that the category $\subcat{ \mathcal{E}^{(n)}_{4}\boxtimes \LVec(\Z{4M}})$ realises two such extensions corresponding to this homomorphism. However when $M$ is even, the category $\subcat{ \mathcal{E}^{(n)}_{4}\boxtimes \LVec(\Z{4M}})$ only realises one such extension. We now construct the other.
 
Appendix~\ref{app:rules} shows that the categories $\mathcal{E}^{(n)}$ contain an order 2 invertible object $\mathbf{4}$. Using \cite[Corollary 3.30]{MR3039775} we see that the subcategory generated by this object has a lift to a copy of $\sVec$ in the centre. Therefore by Lemma~\ref{lem:liftyourself} the subcategory $\langle \mathbf{4} \boxtimes 4M \rangle$ of $ \mathcal{E}_{4}^{(n)} \boxtimes \IVec(\Z{8M})$ has a lift to a copy of $\Rep(\Z{2})$ in the centre. Thus the subcategory $\langle \mathbf{4} \boxtimes 4M \rangle$ of $ \subcat{\mathcal{E}_{4}^{(n)} \boxtimes \IVec(\Z{8M})}$ also has a lift to a copy of $\Rep(\Z{2})$ in the centre. Hence we can de-equivariantize $\subcat{ \mathcal{E}_{4}^{(n)} \boxtimes \IVec(\Z{8M})}$ by the subcategory $\langle  \mathbf{4} \boxtimes 4M \rangle$, to get the fusion category
\[    \subcat{ \mathcal{E}_{4}^{(n)} \boxtimes \IVec(\Z{8M})}_{\langle \mathbf{4} \boxtimes 4M \rangle}.  \] 
We claim that this fusion category is a $\Z{4M}$-graded extension of $\ad(A_7^{(n)})$.
\begin{lemma}
The category $\subcat{ \mathcal{E}_{4}^{(n)} \boxtimes \IVec(\Z{8M})}_{\langle \mathbf{4} \boxtimes 4M \rangle}$ is a $\Z{4M}$-graded extension of $\ad(A_7^{(n)})$.
\end{lemma}
\begin{proof}
A direct computation shows that the four objects 
\[  (\mathbf{1}\boxtimes 0) \oplus (\mathbf{4}\boxtimes 4M) , (\mathbf{2}\boxtimes 0) \oplus (\mathbf{3}\boxtimes 4M) , (\mathbf{3}\boxtimes 0) \oplus (\mathbf{2}\boxtimes 4M), \text{ and } (\mathbf{4}\boxtimes 0) \oplus (\mathbf{1}\boxtimes 4M)\]
have $\ad(A_7)$ fusion rules, and generate the adjoint subcategory of $\subcat{ \mathcal{E}_{4}^{(n)} \boxtimes \IVec(\Z{8M})}_{\langle \mathbf{4} \boxtimes 4M \rangle}$. Thus by Lemma~\ref{lem:adclass} we have a monoidal equivalence
\[   \ad\left( \subcat{ \mathcal{E}_{4}^{(n)} \boxtimes \IVec(\Z{8M})}_{\langle \mathbf{4} \boxtimes 4M \rangle}\right) \to \ad(A_7^{(m)}) \]
for some $m\in \Z{8}^\times / \{\pm\}$. Considering categorical dimensions shows that $m = n$.
\end{proof}

We now need to show that when $M$ is even, the categories $\subcat{ \mathcal{E}_{4}^{(n)} \boxtimes \IVec(\Z{8M})}_{\langle \mathbf{4} \boxtimes 4M \rangle}$ and $\subcat{ \mathcal{E}_{4}^{(n)} \boxtimes \LVec(\Z{4M})}$ are non-equivalent, even up to twisting the associator. 

\begin{lemma}\label{lem:e4diff}
The categories $\subcat{ \mathcal{E}_{4}^{(n)} \boxtimes \IVec(\Z{16M})}_{\langle \mathbf{4} \boxtimes 8M \rangle}$ and $\subcat{ \mathcal{E}_{4}^{(n)} \boxtimes \LVec(\Z{8M})}$ are monoidally inequivalent, even up to twisting the associator.
\end{lemma}
\begin{proof}
The invertible objects of the category $\subcat{ \mathcal{E}_{4}^{(n)} \boxtimes \LVec(\Z{8M})}$ form a group isomorphic to $\Z{2}\times \Z{2M}$, where as the invertible elements of $\subcat{ \mathcal{E}_{4}^{(n)} \boxtimes \IVec(\Z{16M})}_{\langle \mathbf{4} \boxtimes 8M \rangle}$ form a group isomorphic to $\Z{4M}$. These two groups are non-isomorphic, hence the fusion rings of the categories 
\[ \subcat{ \mathcal{E}_{4}^{(n)} \boxtimes \IVec(\Z{16M})}_{\langle \mathbf{4} \boxtimes 8M \rangle}\text{ and } \subcat{ \mathcal{E}_{4}^{(n)} \boxtimes \LVec(\Z{8M})}\]
 are different. Thus these two categories are monoidally inequivalent, even up to twisting the associator.
\end{proof}

Putting everything together we can now prove the classification result of this subsection.
\begin{lemma}\label{lem:a7}
Fix $n\in \Z{8}^\times / \{\pm \}$. If $\cC$ is a $\Z{M}$-graded extension of $\ad(A_{7}^{(n)})$, $\otimes$-generated by an object of Frobenius-Perron dimension less than 2, then either $M$ is even and, up to twisting the associator of $\cC$ by an element of $H^3(\Z{M}, \mathbb{C}^\times)$, the category $\cC$ is monoidally equivalent to either 
\[
\subcat{ A_7^{(n)} \boxtimes \LVec(\Z{M}) }, \text{ or }
\]
\[ \subcat{ A_7^{(n)} \boxtimes \IVec(\Z{2M}) }_{\langle f^{(6)}\boxtimes M \rangle},  \]
or $4$ divides $M$ and, up to twisting the associator of $\cC$ by an element of $H^3(\Z{M}, \mathbb{C}^\times)$, the category $\cC$ is monoidally equivalent to
\[   \subcat{\mathcal{E}^{(n)}_{4} \boxtimes \LVec(\Z{M})},\] 
or $8$ divides $M$ and, up to twisting the associator of $\cC$ by an element of $H^3(\Z{M}, \mathbb{C}^\times)$, the category $\cC$ is monoidally equivalent to
\[   \subcat{\mathcal{E}^{(n)}_{4} \boxtimes \IVec(\Z{2M})}_{ \langle \mathbf{4} \boxtimes M \rangle}.\]
\end{lemma}
\begin{proof}
As usual, we begin by classifying homomorphisms $\Z{M} \to \BrPic(\ad(A_7^{(n)}))$ that may give rise to an extension, $\otimes$-generated by an object of Frobenius-Perron dimension less than 2. We find that there are four families of homomorphisms to consider. When $M$ is even we have the two homomorphisms defined by
\begin{align*}
1 &\mapsto A_7^\text{odd} ,\text{ and } \\
1 & \mapsto D_5^\text{odd}.
\end{align*}
When $M$ is divisible by 4, we have the two homomorphisms defined by
\begin{align*}
1 &\mapsto _{f^{(2)} \leftrightarrow f^{(4)}}A_7^\text{odd} ,\text{ and } \\
1 & \mapsto _{f^{(2)} \leftrightarrow f^{(4)}}D_5^\text{odd}.
\end{align*}
Furthermore, an application of \cite{1711.00645} shows that in order to get a representative of each monoidal equivalence class of the extensions, we in fact only have to consider the two homomorphisms
\begin{align*}
1 &\mapsto A_7^\text{odd} ,\text{ and } \\
1 &\mapsto _{f^{(2)} \leftrightarrow f^{(4)}}A_7^\text{odd}.
\end{align*}
We thus break the proof up in to two cases.

\begin{trivlist}\leftskip=2em
\item \textbf{Case $1 \mapsto A_7^\text{odd}:$}

With this choice of homomorphism we compute that
\[ H^2(\Z{M} , \Inv(\cZ(\ad(A_7^{(n)})))) = \Z{2}.\]
Thus, up to twisting the associator, there are two $\Z{M}$-graded extensions of $\ad(A_7^{(n)})$ corresponding to this choice of homomorphism. These two extensions are realised by the categories
\[ \subcat{ A_7^{(n)} \boxtimes \LVec(\Z{M}) } \text{ and } \subcat{ A_7^{(n)} \boxtimes \IVec(\Z{2M}) }_{\langle f^{(6)}\boxtimes M \rangle}. \]
These categories are $\Z{M}$-graded extensions of $\ad(A_7^{(n)})$ by Lemma~\ref{lem:aOddExt}, and they are monoidally distinct, even up to twisting the associator, by Lemma~\ref{lem:aOddDiff}.

\vspace{1em}

\item \textbf{Case $1 \mapsto _{f^{(2)} \leftrightarrow f^{(4)}}A_7^\text{odd}$:}

With this choice of homomorphism we also compute that
\[ H^2(\Z{M} , \Inv(\cZ(\ad(A_7^{(n)})))) = \Z{2}.\] 
Thus, up to twisting the associator, there are two $\Z{M}$-graded extensions of $\ad(A_7^{(n)})$ corresponding to this choice of homomorphism. When $M$ is not divisible by 8, Lemma~\ref{lem:a7ext} along with Corollary~\ref{lem:dobgencov} imply that both these extensions are realised by the single category
\[  \subcat{\mathcal{E}^{(n)}_{4} \boxtimes \LVec(\Z{M})}.   \]
When $M$ is divisible by 8, Lemmas~\ref{lem:a7ext} and \ref{lem:e4diff} imply that these two extensions are realised by the categories
\[ \subcat{\mathcal{E}^{(n)}_{4} \boxtimes \LVec(\Z{M})} \text{ and } \subcat{\mathcal{E}^{(n)}_{4} \boxtimes \IVec(\Z{2M})}_{ \langle \mathbf{4} \boxtimes M \rangle}.   \] 
\end{trivlist}
\end{proof}
\section{A proof of the main classification result}\label{sec:mainproof}

In this section we tie the previous results of this paper together to give a proof of Theorem~\ref{thm:main}.

\begin{proof}[Proof of Theorem~\ref{thm:main}]
Let $\cC$ be a fusion category, $\otimes$-generated by an object $X$ of Frobenius-Perron dimension less than 2, such that $\ad(\cC) = \langle X\otimes X^* \rangle$. Then by Theorem~\ref{thm:ADEext} the category $\cC$ is monoidally equivalent to a $\Z{M}$-graded extension of one of the following fusion categories:
\begin{align*}
\ad(A_N^{(n)}) & \text{ for } n\in \Z{N+1}^\times / \{\pm\}, \\
\ad(D_{2N}^{(n)}) & \text{ for } n\in \Z{4N-3}^\times / \{\pm\}, \\
\ad(E_6^{(n,\pm)}) & \text{ for } n\in \Z{12}^\times / \{\pm\}, \\
\ad(E_8^{(n,\pm)}) & \text{ for } n\in \Z{30}^\times/ \{\pm\}.
\end{align*}
The $\Z{M}$-graded extensions of these categories are classified, up to twisting the associator by an element of $H^3(\Z{M},\mathbb{C}^\times)$, in Lemmas~\ref{lem:aeven},~\ref{lem:aodd},~\ref{lem:a3},~\ref{lem:a7},~\ref{lem:d2n},~\ref{lem:d4},~\ref{lem:d10},~\ref{lem:e6}, and ~\ref{lem:e8}. Thus $\cC$ is monoidally equivalent, up to twisting the associator, to one of the categories in the statements of these Lemmas. Several changes of variables gets us to the statement of Theorem~\ref{thm:main}.
\end{proof}

\appendix
\section{Fusion Rules for the categories $\mathcal{E}^{(n)}_4$ and $\mathcal{E}^{(n,\pm)}_{16,6}$ }\label{app:rules}
In this appendix we compute the fusion rules for the categories $\mathcal{E}^{(n)}_4$ and $\mathcal{E}^{(n,\pm)}_{16,6}$. We begin with the $\mathcal{E}^{(n,\pm)}_{16,6}$ case, which is conceptually easier, but computationally harder.

Recall from Lemma~\ref{lem:D10ext} that the categories $\mathcal{E}^{(n,\pm)}_{16,6}$ are $\Z{6}$-graded extensions of $\ad(D_{10}^{(n)})$. The $3$-graded piece under this grading is either the bimodule $D_{10}^{\text{odd}}$ or $_{P\leftrightarrow Q}D_{10}^{\text{odd}}$. Thus $\mathcal{E}^{(n,\pm)}_{16,6}$ has a $\Z{2}$-graded subcategory equivalent to $D_{10}^{(\pm n , \pm)}$, and $\mathcal{E}^{n,m,\pm}_{16,6}$ is thus a $\Z{3}$-graded extension of $D_{10}^{(\pm n , \pm)}$.

The Brauer-Picard group of $D_{10}^{(\pm n , \pm)}$ was computed in \cite{MR3808052}, and the only order 3 bimodules are $_{P\leftrightarrow Q}E_7$ and $_{P\leftrightarrow Q}\overline{E}_7$. Thus, as a $\Z{3}$-graded extension of $D_{10}^{(\pm n , \pm)}$, the 1 and 2 graded pieces consist of these two bimodules. This gives us fusion rules for tensoring a $D_{10}$ object with either a $_{P\leftrightarrow Q}E_7$ object, or a $_{P\leftrightarrow Q}\overline{E_7}$ object.

We now aim to determine how two $_{P\leftrightarrow Q}E_7$ objects tensor. Due to the grading, the tensor of two such object must live in the $_{P\leftrightarrow Q}\overline{E_7}$ piece. Let $X$ and $Y$ be $_{P\leftrightarrow Q}E_7$  objects, and $Z$ a $D_{10}$ object, then associativity of the fusion rules gives us that 
\[    (X \otimes Z)\otimes Y \cong X \otimes (Z\otimes Y).\]
Along with the fact that tensor product must preserve Frobenius-Perron dimensions, this allows us to completely determine the fusion rules for tensoring two objects in the $1$-graded piece.

By considering Frobenius-Perron dimensions, and the grading of the category, we can completely determine the dual of each object. Each $D_{10}$ object is self-dual, and each $_{P\leftrightarrow Q}E_7$ is dual to the unique object in the $_{P\leftrightarrow Q}\overline{E_7}$ piece with the same Frobenius-Perron dimension. 

Finally, an application of Frobenius reciprocity gives fusion rules for the entire category. The full fusion rules for the categories $\mathcal{E}^{(n,\pm)}_{16,6}$ can be found at the authors website \href{http://cainedie.com/E166fusion.txt}{http://cainedie.com/E166fusion.txt}. To save space we only present in Figure~\ref{fig:Spaceship} the fusion graph for the $\otimes$-generating object of Frobenius-Perron dimension $2\cos(\frac{\pi}{18})$, living in the $1$-graded piece. 
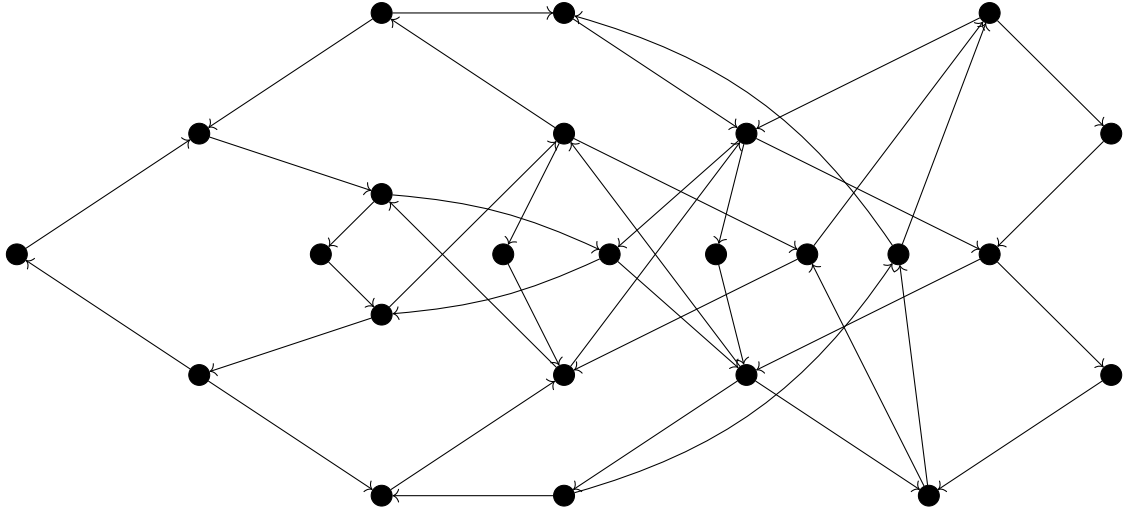
\begin{figure}
\centering
\begin{center}\begin{tikzpicture}[scale = .8]
    \node[shape=circle,draw=black, fill = black, scale = .707] (1) at (0,0) {};
    \node[shape=circle,draw=black, fill = black, scale = .707] (11) at (3,2) {};
    \node[shape=circle,draw=black, fill = black, scale = .707] (18) at (3,-2) {};
    \node[shape=circle,draw=black, fill = black, scale = .707] (10) at (6,4) {};
    \node[shape=circle,draw=black, fill = black, scale = .707] (9) at (6,-4) {};
    \node[shape=circle,draw=black, fill = black, scale = .707] (21) at (6,1) {} ;
    \node[shape=circle,draw=black, fill = black, scale = .707] (14) at (6,-1) {} ;
 \node[shape=circle,draw=black , fill = black, scale = .707] (2) at (5,0) {} ;
 \node[shape=circle,draw=black, fill = black, scale = .707] (8) at (9.75,0) {} ;
 \node[shape=circle,draw=black, fill = black, scale = .707] (13) at (9,4) {} ;
 \node[shape=circle,draw=black, fill = black, scale = .707] (23) at (9,2) {} ;
 \node[shape=circle,draw=black, fill = black, scale = .707] (16) at (9,-2) {} ;
 \node[shape=circle,draw=black, fill = black, scale = .707] (20) at (9,-4) {} ;
 \node[shape=circle,draw=black, fill = black, scale = .707] (3) at (8,0) {} ;
 \node[shape=circle,draw=black, fill = black, scale = .707] (24) at (12,2) {} ;
 \node[shape=circle,draw=black, fill = black, scale = .707] (17) at (12,-2) {} ;
\node[shape=circle,draw=black, fill = black, scale = .707] (7) at (13,0) {} ;
\node[shape=circle,draw=black, fill = black, scale = .707] (4) at (11.5,0) {} ;
\node[shape=circle,draw=black, fill = black, scale = .707] (5) at (14.5,0) {} ;
\node[shape=circle,draw=black, fill = black, scale = .707] (6) at (16,0) {} ;
\node[shape=circle,draw=black, fill = black, scale = .707] (15) at (16,4) {} ;
\node[shape=circle,draw=black, fill = black, scale = .707] (22) at (15,-4) {} ;
\node[shape=circle,draw=black, fill = black, scale = .707] (19) at (18,2) {} ;
\node[shape=circle,draw=black, fill = black, scale = .707] (12) at (18,-2) {} ;

    \path [->] (1) edge node {$$} (11);
    \path [->](18) edge node {$$} (1);
\path [->](11) edge node {} (21);
\path [->](10) edge node {} (11);
\path [->](18) edge node {} (9);
\path [->](14) edge node {} (18);
\path [->](21) edge node {} (2);
\path [->](2) edge node {} (14);
\path [->](21) edge[bend left = 10] node {} (8);
\path [->](8) edge[bend left = 10] node {} (14);
\path [->](14) edge node {} (23);
\path [->](16) edge node {} (21);
\path [->](9) edge node {} (16);
\path [->](20) edge node {} (9);
\path [->](23) edge node {} (10);
\path [->](10) edge node {} (13);   
\path [->](23) edge node {} (3);
\path [->](3) edge node {} (16);
\path [->](13) edge node {} (24);
\path [->](8) edge node {} (17);
\path [->](24) edge node {} (8);
\path [->](5) edge[bend right = 20] node {} (13);
\path [->](20) edge[bend right = 20] node {} (5);
\path [->](17) edge node {} (20);
\path [->](23) edge node {} (7);
\path [->](17) edge node {} (23);
\path [->](16) edge node {} (24);
\path [->](7) edge node {} (16);
\path [->](24) edge node {} (4);
\path [->](4) edge node {} (17);
\path [->](15) edge node {} (24);
\path [->](24) edge node {} (6);
\path [->](17) edge node {} (22);
\path [->](6) edge node {} (17);
\path [->](5) edge node {} (15);
\path [->](7) edge node {} (15);
\path [->](22) edge node {} (5);
\path [->](22) edge node {} (7);
\path [->](15) edge node {} (19);
\path [->](12) edge node {} (22);
\path [->](6) edge node {} (12);
\path [->](19) edge node {} (6);

\end{tikzpicture}\end{center}
\caption{Fusion graph for the $\otimes$-generating object of $\mathcal{E}^{(n,\pm)}_{16,6}$ of Frobenius-Perron dimension $2\cos(\frac{\pi}{18})$} \label{fig:Spaceship}
\end{figure}

We now compute the $\mathcal{E}^{(n)}_4$ case. Recall the category $\mathcal{E}^{(n)}_4$ is a $\Z{4}$-graded extension of a category of adjoint $A_7$ type, with the $1$, $2$, and $3$ graded pieces being the bimodules $_{f^{(2)} \leftrightarrow f^{(4)} }A_7^\text{odd}$, $D_5^\text{even}$, and $_{f^{(2)} \leftrightarrow f^{(4)} }D_5^\text{even}$ respectively. 

From \cite{Su4} we know the Frobneius-Perron dimensions of the 12 simple objects of $\mathcal{E}^{(n)}_4$. They are
\[  \left\{1,\sqrt{2}+1,\sqrt{2}+1,1,\sqrt{\sqrt{2}+2},\sqrt{\sqrt{2}+2},\sqrt{2 \left(\sqrt{2}+2\right)},\sqrt{2},\sqrt{2}+2,\sqrt{\sqrt{2}+2},\sqrt{\sqrt{2}+2},\sqrt{2 \left(\sqrt{2}+2\right)}\right\}. \]
With the first 4 objects living in the 0-graded piece, the next 3 objects living in the 1 graded piece, the next 2 objects living in the $2$ graded piece, and the final 3 objects living in the $3$-graded piece. For simplicity we call the 12 objects of this category the numbers $\mathbf{1}$ through $\mathbf{12}$.

Simply by considering Frobenius-Perron dimensions, along with the associativity check from the $\mathcal{E}_{16,6}$ case, allows us to completely determine fusion rules for all objects apart from fusion between the objects $5,6$ and $10,11$. Here we get four possible fusion rules:
\begin{center}
    \begin{tabular}{l | c | c | c | c }
			Fusion Rule & $\mathbf{5}\otimes \mathbf{10}$ & $\mathbf{5}\otimes \mathbf{11}$ & $\mathbf{6}\otimes \mathbf{10}$ & $\mathbf{6}\otimes \mathbf{11}$ \\
			\midrule
	                  1                      & $ \mathbf{1} \oplus \mathbf{2} $ & $ \mathbf{3} \oplus \mathbf{4}$   &   $\mathbf{3} \oplus \mathbf{4} $ &  $\mathbf{1} \oplus \mathbf{2}$ \\
	                   2                     & $ \mathbf{1}\oplus \mathbf{3} $ & $ \mathbf{2} \oplus \mathbf{4}$    &   $\mathbf{2} \oplus \mathbf{4} $ &  $\mathbf{1} \oplus \mathbf{3}$  \\			
	                   3                     & $ \mathbf{3} \oplus \mathbf{4} $ & $ \mathbf{1} \oplus \mathbf{2}$   &   $\mathbf{1} \oplus \mathbf{2} $ &  $\mathbf{3} \oplus \mathbf{4}$   \\
	                   3                     & $ \mathbf{2} \oplus \mathbf{4} $ & $ \mathbf{1} \oplus \mathbf{3}$   &   $\mathbf{1} \oplus \mathbf{3} $ &  $\mathbf{2} \oplus \mathbf{4}$  \\
    	\bottomrule
    \end{tabular}
\end{center}
However a direct computation shows each of these four fusion rules are isomorphic. Thus we can completely determine the fusion rules for the categories $\mathcal{E}^{(n)}_4$. Again we only present the fusion graph for the generating object $\mathbf{5}$ of Frobenius-Perron dimension $\sqrt{2+\sqrt{2}}$ in Figure~\ref{fig:Teepee}. Full fusion rules can be found at the authors website \href{http://cainedie.com/E4fusion.txt}{http://cainedie.com/E4fusion.txt}.

\begin{figure}
\centering
\begin{center}\begin{tikzpicture}[scale = .6]
    \node[shape=circle,draw=black, fill = black, scale = .707] (1) at (0,0) {};
    \node[shape=circle,draw=black, fill = black, scale = .707] (5) at (4,0) {};

    \node[shape=circle,draw=black, fill = black, scale = .707] (9) at (8,6) {};

  \node[shape=circle,draw=black, fill = black, scale = .707] (6) at (12,0) {};
    \node[shape=circle,draw=black, fill = black, scale = .707] (4) at (16,0) {};

\node[shape=circle,draw=black, fill = black, scale = .707] (10) at (2,3) {};
\node[shape=circle,draw=black, fill = black, scale = .707] (11) at (14,3) {};

\node[shape=circle,draw=black, fill = black, scale = .707] (2) at (4,6) {};
\node[shape=circle,draw=black, fill = black, scale = .707] (3) at (12,6) {};

\node[shape=circle,draw=black, fill = black, scale = .707] (7) at (6,9) {};
\node[shape=circle,draw=black, fill = black, scale = .707] (12) at (10,9) {};

\node[shape=circle,draw=black, fill = black, scale = .707] (8) at (8,12) {};

    \path [->] (1) edge node {$$} (5);
    \path [->](5) edge node {$$} (9);
 \path [->](6) edge node {$$} (9);
 \path [->](4) edge node {$$} (6);
 \path [->](10) edge node {$$} (1);
 \path [->](11) edge node {$$} (4);
 \path [->](9) edge node {$$} (10);
 \path [->](9) edge node {$$} (11);
 \path [->](10) edge node {$$} (2);
 \path [->](2) edge node {$$} (6);
\path [->](3) edge node {$$} (5);
 \path [->](11) edge node {$$} (3);
 \path [->](2) edge node {$$} (7);
 \path [->](7) edge node {$$} (8);
 \path [->](12) edge node {$$} (3);
 \path [->](8) edge node {$$} (12);
 \path [->](7) edge node {$$} (9);
 \path [->](3) edge node {$$} (7);
 \path [->](12) edge node {$$} (2);
 \path [->](9) edge node {$$} (12);
\end{tikzpicture}\end{center}

\caption{Fusion graph of the object $\mathbf{5} \in \mathcal{E}^{(n)}_{4}$} \label{fig:Teepee}
\end{figure}
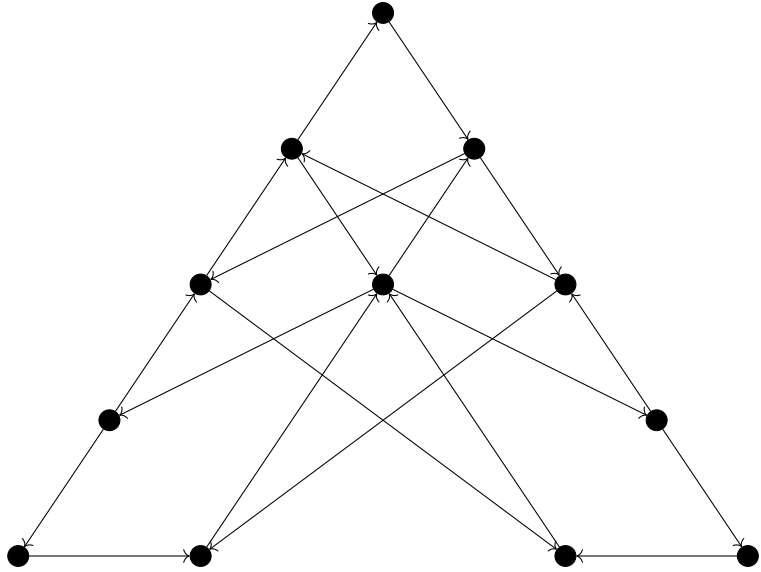

Using the fusion rules of $\mathcal{E}_4$, we can directly compute the fusion ring automorphisms.
\begin{lemma}\label{lem:E4auto}
There exist three non-trivial fusion ring automorphisms of $\mathcal{E}^{(n)}_4$ given by
\[ 1)   \mathbf{5} \leftrightarrow \mathbf{6} ,  \mathbf{10} \leftrightarrow \mathbf{11}\]
\[ 2)   \mathbf{2} \leftrightarrow \mathbf{3} ,  \mathbf{5} \leftrightarrow \mathbf{10} ,  \mathbf{6} \leftrightarrow \mathbf{11},  \mathbf{7} \leftrightarrow \mathbf{12}   \]
and
\[ 3)   \mathbf{2} \leftrightarrow \mathbf{3} ,  \mathbf{5} \leftrightarrow \mathbf{11} ,  \mathbf{6} \leftrightarrow \mathbf{10}, \mathbf{7} \leftrightarrow \mathbf{12}   \]
\end{lemma}
While this Lemma may seem out of place, it proves important in showing that the categories $\mathcal{E}_4^{(n)}$ realises two different $\Z{4}$-graded extensions of the category $\ad(A_7^{(n)})$.
\bibliography{bibliography} 
\bibliographystyle{plain}
\end{document}